\documentclass{amsart}
\usepackage{amsmath,amsfonts,amsthm,amssymb,stmaryrd,comment}
\usepackage{xcolor}

\renewcommand{\leq}{\leqslant}
\renewcommand{\geq}{\geqslant}

\newcommand{\ind}[1]{1_{#1}}

\newcommand{\bbz}{\mathbb{Z}}
\newcommand{\bbr}{\mathbb{R}}

\newcommand{\bbc}{\mathbb{C}}

\newcommand{\ls}{\lesssim}
\newcommand{\gs}{\gtrsim}

\newcommand*{\bbe}{
  \mathop{
    \mathchoice{\vcenter{\hbox{\larger[4]$\mathbb{E}$}}}
               {\kern0pt\mathbb{E}}
               {\kern0pt\mathbb{E}}
               {\kern0pt\mathbb{E}}
  }\displaylimits
}

\newcommand{\abs}[1]{\left\lvert #1\right\rvert}
\newcommand{\Abs}[1]{\lvert #1\rvert}
\newcommand{\brac}[1]{\left( #1\right)}

\newcommand{\norm}[1]{\left\lVert #1\right\rVert}

\newtheorem{theorem}{Theorem}
\newtheorem{lemma}{Lemma}

\newtheorem{corollary}{Corollary}
\newtheorem{conjecture}{Conjecture}

\theoremstyle{definition}
\newtheorem{definition}{Definition}
\begin{document}

\title{Control and its applications in additive combinatorics}
\author{Thomas F. Bloom}
\address{Department of Mathematics, University of Manchester, Manchester, M13 9PL}
\email{thomas.bloom@manchester.ac.uk}

\begin{abstract}
We prove new quantitative bounds on the additive structure of sets obeying an $L^3$ `control' assumption, which arises naturally in several questions within additive combinatorics. This has a number of applications -- in particular we improve the known bounds for the sum-product problem, the Balog-Szemer\'{e}di-Gowers theorem, and the additive growth of convex sets.
\end{abstract}

\maketitle
\section{Introduction}
Everything in this paper, unless specified otherwise, takes place in some fixed (but arbitrary) abelian group $G$. All sets will be assumed to be finite, and all functions assumed to have finite support. We will investigate the following concept, and highlight its relevance to several questions in additive combinatorics.
\begin{definition}[Control]
Let $\kappa\in (0,1]$. We say that a set $A$ has control $\kappa$ if $\kappa$ is minimal such that, for any finite set $B$,\footnote{For basic definitions and conventions about e.g. what the convolution operator $\ast$ means see Section~\ref{sec-not}.}
\[\sum_x 1_A\ast 1_B(x)^3\leq \kappa\abs{A}^{2}\abs{B}^{2}.\]
\end{definition}

For comparison, the trivial bounds $\norm{1_A\ast 1_B}_\infty\leq \min(\abs{A},\abs{B})$ and $\sum_x 1_A\ast 1_B(x)=\abs{A}\abs{B}$ imply  that every set has control $\leq 1$. We also have the trivial lower bound that, if $A$ has control $\kappa$,
\[\kappa\geq \abs{A}^{-1},\]
which can be seen by taking $B=-A$, for example, and noting that $1_A\ast 1_{-A}(0)=\abs{A}$. This lower bound is achieved (up to logarithmic factors) by, for example, convex sets of real numbers, such as geometric progressions or $\{n^2: 1\leq n\leq N\}$ (as we will prove in Lemma~\ref{lem-convcont}). It may also be useful to note that if $G$ is a finite group and $A\subseteq G$ with density $\alpha=\abs{A}/\abs{G}$ then (taking $B=G$) we have $\kappa \geq \alpha$. 

This definition may appear strange to those unfamiliar with the area, but this kind of control property arises independently in the study of several central questions within additive combinatorics. The study of such a property was pioneered by Shkredov in a number of papers, particularly \cite{Sh13}. Since this kind of $L^3$ control hypothesis is often produced from an application of the Szemer\'{e}di-Trotter theorem, Shkredov has referred to sets with such a property as Szemer\'{e}di-Trotter sets. We take this opportunity to coin the new term `control' for the relevant $\kappa$, to highlight the importance of this parameter. (Moreover, in at least some applications control is present for reasons unrelated to the Szemer\'{e}di-Trotter theorem, so we felt a more neutral term was appropriate.) Of course, what we call control should more properly be called $L^3$-control, and one can imagine natural variants for other norms. We hope to return to such generalisations and their applications in future work. (In particular the $L^4$ analogue would be relevant for sum-product problems in $\mathbb{F}_p$, arising from the use of the point-line incidence results of Stevens and de Zeeuw \cite{StdZ17} in place of the Szemer\'{e}di-Trotter theorem.)

In this paper we will be concerned with, amongst other things, how control relates to more usual notions of additive structure such as the size of the sumset $A+A$ and the additive energy, defined by
\[E(A) = \norm{1_A\ast 1_A}_2^2=\sum_x 1_A\ast 1_A(x)^2.\]
It is an immediate trivial consequence of the Cauchy-Schwarz inequality that
\[E(A) \leq \brac{\sum 1_A\ast 1_A(x)^3}^{1/2}\brac{\sum 1_A\ast 1_A(x)}^{1/2}\leq \kappa^{1/2}\abs{A}^3.\]
Similarly, trivial applications of H\"{o}lder's inequality yield $\abs{A+A}\geq \kappa^{-1/2}\abs{A}$ and $\abs{A-A}\geq \kappa^{-1/2}\abs{A}$. These bounds should be thought of as the `trivial' bounds, and they do not make full use of the power of the control assumption. Shkredov was the first to prove bounds which go beyond these thresholds. In this paper we will prove the following non-trivial bounds, which are improvements over those previously available for $E(A)$ and $\abs{A-A}$.

\begin{theorem}\label{th-main}
If $A$ has control $\kappa$ then, for any $\epsilon>0$,

\[E(A) \ll_\epsilon \kappa^{27/50-\epsilon}\abs{A}^3,\]
\[\abs{A+A}\gg_\epsilon \kappa^{-11/19+\epsilon}\abs{A},\]
and
\[\abs{A-A} \gg_\epsilon \kappa^{-2506/4175+\epsilon}\abs{A}.\]
\end{theorem}
Note that $27/50=0.54$, $11/19\approx 0.578$, and $2506/4175\approx 0.6002$. The previous best exponents available in the literature (at least implicitly) were $7/13\approx 0.538$, $11/19$, and $3/5=0.6$ respectively (due to Shkredov \cite{Sh13}, Rudnev and Stevens \cite{RuSt22}, and Schoen and Shkredov \cite{ScSh11}). We defer further discussion of the previous literature and such results to Section~\ref{sec-control}.

The reader will observe that our quantitative improvements are small, and in the case of $\abs{A+A}$ non-existent. One of our main goals is to show that some improvement is possible, and moreover using purely `elementary' techniques (nothing deeper than the pigeonhole principle and H\"{o}lder's inequality). The original investigations of Shkredov \cite{Sh13} used a more elaborate spectral method (in particular studying eigenvalues of the convolution operator). (We note however that the energy bound $E(A)\ll_\epsilon \kappa^{7/13-\epsilon}\abs{A}^{3}$ was reproved using elementary methods by Olmezov \cite{Ol21}, whose approach has many similarities to our own.)

Our second goal is to advertise the concept of control and explain how it is intertwined with several important questions within additive combinatorics, bringing together various scattered threads in the literature. We hope that by demonstrating the utility of the type of quantitative improvements in Theorem~\ref{th-main} to other quantitative questions we can encourage further research.

There are a number of applications of control within additive combinatorics, and we present the highlights (incorporating improvements arising from the new bounds in Theorem~\ref{th-main}) below. This paper has been structured to make the applications easy to export -- in particular, the reader should note that improvements to the exponents in Theorem~\ref{th-main} immediately propagate to corresponding improvements on all of the applications listed below. (In particular, and surprisingly, as Shakan \cite{Sh19} has noted, one can obtain new bounds for the sum-product problem via the purely additive problem of improving the bounds in Theorem~\ref{th-main}.)

\subsection{Application 1: Convex sets}

The most direct application of control is to the additive behaviour of convex sets. A convex set is a finite set $A=\{a_1,\ldots,a_n\}\subset \bbr$ such that $a_{i+1}-a_i> a_{i}-a_{i-1}$ for $1<i<n$. The relevance of control to convex sets is simple: convex sets have the best possible control of $\kappa \ls \abs{A}^{-1}$. This can be proved using the Szemer\'{e}di-Trotter incidence theorem, as first observed by Schoen and Shkredov \cite{ScSh11}. We include a proof in this paper as Lemma~\ref{lem-convcont}.

Convex sets should exhibit non-trivial growth under addition (first proved qualitatively by Hegyv\'{a}ri \cite{He86}). It has long been conjectured (in particular by Erd\H{o}s \cite{Er77}) that convex sets should exhibit maximal growth, in that $\min(\abs{A+A},\abs{A-A})\gg_\epsilon \abs{A}^{2-\epsilon}$ and $E(A)\ll_\epsilon \abs{A}^{2+\epsilon}$. Theorem~\ref{th-main}, coupled with the fact that convex sets have control $\ls \abs{A}^{-1}$, immediately implies the following improvements for two out of the three of these measures of additive growth.
\begin{theorem}\label{th-convex}
If $A\subset \bbr$ is a convex set then, for any $\epsilon>0$,
\[E(A) \ll_\epsilon \abs{A}^{123/50+\epsilon},\]
and
\[\abs{A-A} \gg_\epsilon \abs{A}^{6681/4175-\epsilon}.\]
\end{theorem}
The best previous bounds were $E(A)\ll_\epsilon \abs{A}^{32/13+\epsilon}$ (due to Shkredov \cite{Sh13}) and $\abs{A-A}\gs \abs{A}^{8/5}$ (due to Schoen and Shkredov \cite{ScSh11}). For comparison note that $32/13\approx 2.461$ and $8/5=1.6$, while $123/50=2.46$ and $6681/4175\approx 1.6002$.

For $\abs{A+A}$ we have been unable to improve on the bound $\abs{A+A}\gg_\epsilon \abs{A}^{30/19-\epsilon}$ of Rudnev and Stevens \cite{RuSt22}. Various improvements to these bounds are known if stronger convexity hypotheses are assumed, as shown by Olmezov \cite{Ol20}. For example, if the convexity assumption is strengthened to a natural fourth-order variant then Olmezov proved $E(A) \ll \abs{A}^{7/3}$ and $\abs{A+A}\gg \abs{A}^{5/3}$. Other improvements under a similar strengthened convexity assumption were proved by Bradshaw, Hanson, and Rudnev \cite{BHR22}.

One can ask, more generally, about the relationship between the additive growth of $A$ and $f(A)$, where $f:\bbr\to\bbr$ is a convex function (the case of convex $A$ is the special case where $A=f(\{1,\ldots,\abs{A}\})$). This was investigated by Shkredov \cite{Sh15} and, most recently, by Stevens and Warren \cite{StWa22}. As a consequence of our methods we deduce the following.
\begin{theorem}\label{th-conv2}
Let $A\subset \bbr$ be a finite set and $f:\bbr \to \bbr$ be a continuous convex function. For any $\epsilon>0$
\[\abs{A+A}\abs{f(A)+f(A)}\gg_\epsilon \abs{A}^{49/19-\epsilon}\]
and
\[\abs{A-A}\abs{f(A)-f(A)}\gg_\epsilon \abs{A}^{10856/4175-\epsilon}.\]
In particular,
\[\max( \abs{A+A}, \abs{f(A)+f(A)}) \gg_\epsilon \abs{A}^{49/38-\epsilon}\]
and
\[\max(\abs{A-A}, \abs{f(A)-f(A)})\gg_\epsilon \abs{A}^{5428/4175-\epsilon}.\]
\end{theorem}
The inequality for sumsets is equal in strength to that proved by Stevens and Warren \cite{StWa22}. The conclusion for difference sets is slightly stronger than that of Stevens and Warren, who obtained an exponent of $13/10=1.3$, while $5428/4175\approx 1.3001$. (For comparison note that $49/38\approx 1.289$.) In fact Stevens and Warren prove a more general version, allowing for an inequality on $\abs{A+B}\abs{f(A)+g(B)}$ for possibly different sets $A$ and $B$ and different (convex) $f$ and $g$, which also allows one to bound $\abs{A+f(A)}$. Our methods can in principle obtain such asymmetric conclusions, but we have chosen to focus on the symmetric case to simplify the exposition.

\subsection{Application 2: The sum-product problem}
The sum-product problem is one of the central problems within additive combinatorics: given an arbitrary set of real numbers (say), must either the sum set or the product set grow? In other words, is it true that no set can be simultaneously additively and multiplicatively structured? There are a number of ways to make this question precise. The connection with control is more involved here, and we use the work of a number of authors below, who have found ingenious ways to maximise the amount of control available when studying sum-product problems. In short, the heuristic is that a set with lots of multiplicative structure must have good additive control -- and hence via the type of estimates in Theorem~\ref{th-main} must grow under addition.

We present small improvements to a number of sum-product problems that have been considered in the literature. These improvements stem almost entirely from the improved exponents of Theorem~\ref{th-main}, although we take this opportunity to sharpen some of the previous arguments, and consolidate these various applications under the unified notion of control. Once again, the improvements here are marginal, but we want to record some of the consequences of the new bounds of Theorem~\ref{th-main}, and to demonstrate the variety of sum-product problems which are intimately connected with the notion of control.

The usual way to precisely frame the sum-product question is to ask for the largest exponent $c$ such that (for any finite $A\subset \bbr$ and $\epsilon>0$)
\[\max(\abs{A+A},\abs{AA})\gg_\epsilon \abs{A}^{c-\epsilon}.\]
The natural conjecture, first posed by Erd\H{o}s, is that $c=2$ is the truth here. An elegant geometric argument of Solymosi \cite{So09} proved that $c=4/3$ is permissible. Konyagin and Shkredov \cite{KoSh15} found a way to combine Solymosi's argument with the notion of control, which gave an improved exponent of $4/3+\delta$ for some small (but explicit) $\delta>0$. This argument in turn saw further refinements, leading to the previous best exponent, due to Rudnev and Stevens \cite{RuSt22}, of
\[c=\frac{1558}{1167}=\frac{4}{3}+\frac{2}{1167}\approx 1.33504.\]
We offer a marginal improvement, and also note that a slightly better exponent is available if we also consider the size of the difference set $\abs{A-A}$.
\begin{theorem}\label{th-spmain}
If $A\subset \bbr$ is any finite set then, for any $\epsilon>0$,
\[\max(\abs{A+A},\abs{AA})\gg_\epsilon \abs{A}^{c_1-\epsilon}\]
and
\[\max(\abs{A-A},\abs{A+A},\abs{AA})\gg_\epsilon  \abs{A}^{c_2-\epsilon}\]
where
\[c_1=\frac{1270}{951}=\frac{4}{3}+\frac{2}{951}\approx 1.33543.\]
and
\[c_2=\frac{144511}{108174}\approx 1.33591.\]
\end{theorem}

One might instead consider the additive and multiplicative energies, in place of the sizes of the sum set and product set. Balog and Wooley \cite{BaWo17} were the first to consider this variant, and proved that, in any finite $A\subset \bbr$, there must exist some large subset of $A$ with either small additive energy or small multiplicative energy. This was improved by Rudnev, Shkredov, and Stevens \cite{RSS20} and then further by Shakan \cite{Sh19}. The best previous estimate in this direction is due to Xue \cite{Xu21}, who proved that any $A\subset \bbr$ contains $A'\subseteq A$ of size $\gg \abs{A}$ such that, for any $\epsilon>0$,
\[\min( E_+(A),E_\times(A))\ll_\epsilon \abs{A}^{736/271+\epsilon}.\]
Combining Theorem~\ref{th-main} with the argument of Xue we obtain the following slight improvement.

\begin{theorem}\label{th-spenergy}
For any finite $A\subset \bbr$ there exists $A'\subseteq A$ with $\abs{A'}\gg \abs{A}$ such that, for any $\epsilon>0$,
\[\min(E_+(A'),E_\times(A')) \ll_\epsilon \abs{A}^{258/95+\epsilon}.\]
\end{theorem}
For comparison, note that $736/271\approx 2.71586$ while $258/95\approx 2.71578$. A construction of Balog and Wooley \cite{BaWo17} shows that the best exponent possible here is $7/3$.

An alternative type of sum-product result (sometimes known as expander results) is to prove a non-trivial lower bound for the size of a set under a fixed operation which involves both addition and multiplication, such as $A(A+A)$. Our method produces the following small improvements to previous expander results.

\begin{theorem}\label{th-spshift}
For any finite $A\subset \bbr$ there exists some $a\in A$ such that, for any $\epsilon>0$,
\[\abs{A(A+a)}\gg_\epsilon \abs{A}^{229/152-\epsilon},\]
\[\abs{A(A+A)}\gg_\epsilon \abs{A}^{32/21-\epsilon}\]
and
\[\abs{A(A-A)}\gg_\epsilon \abs{A}^{7239/4733-\epsilon}.\]
\end{theorem}
The previous best exponents were due to Murphy, Roche-Newton, and Shkredov \cite{MRS17}, and were (respectively) $140/93\approx 1.505$, $184/121\approx 1.521$, and $26/17\approx 1.5294$. For comparison, $229/152\approx 1.507$, $32/21\approx 1.524$ and $7239/4733\approx 1.5295$. (The reader may also be interested in a similar expander result of Rudnev and Stevens \cite{RuSt22}, who show that $\abs{AA+AA}\gg_\epsilon \abs{A}^{1.5875-\epsilon}$ -- our methods currently yield no improvement on this exponent.)

Finally, one can ask for a sum-product result which finds (in an arbitrary finite set of real numbers) some set which lacks both additive and multiplicative structure. We say that $A\subset \bbr$ is an additive Sidon set if the only solutions to $a+b=c+d$ are the trivial ones $\{a,b\}=\{c,d\}$, and similarly define a multiplicative Sidon set. We say that $A$ is a bi-Sidon set if it is simultaenously an additive and multiplicative Sidon set. Ruzsa \cite{Ru06} observed that every $A\subset \bbr$ contains a bi-Sidon set of size $\gg \abs{A}^{1/3}$. Pach and Zakharov \cite{PaZa24} have improved this to $\gg \abs{A}^{11/26}$. The following small improvement is a consequence of our new estimates.
\begin{theorem}\label{th-pachzakh}
Let $\epsilon>0$. Any finite $A\subset \bbr$ contains a bi-Sidon subset of size
\[\gg_\epsilon \abs{A}^{122/285-\epsilon}.\]
\end{theorem}
For comparison, note that $11/26\approx 0.423$ while $122/285\approx 0.428$. The likely true exponent here is $1/2$ -- such a bound would follow from Pach and Zakharov's argument if the upper bound of Theorem~\ref{th-spenergy} could be improved to $\ll_\epsilon \abs{A}^{5/2+\epsilon}$.
\subsection{Application 3: Balog-Szemer\'{e}di-Gowers theorem}
As another application we strengthen the classical Balog-Szemer\'{e}di-Gowers theorem. The subject of this theorem is to find some converse to the trivial fact that $\abs{A-A}\geq \abs{A}^4/E(A)$ -- in other words, the fact that if $A$ has small additive energy then it must have large difference set. The obvious converse fails, since one could for example have half of $A$ being a geometric progression (which creates a near-maximal difference set) and the other half being an arithmetic progression (which creates near-maximal additive energy).

It was first observed by Balog and Szemer\'{e}di \cite{BaSz94} that a converse does, however, hold, provided one is willing to pass to a subset of $A$ first. This was reproved by Gowers \cite{Go01} with polynomial bounds. As a consequence of our investigations into control we are able to improve the quantitative aspect of this polynomial control.

The link between control and the Balog-Szemer\'{e}di-Gowers theorem is new to this paper, although our argument has much in common with that of Schoen \cite{Sc15}.

\begin{theorem}\label{th-bsgmain}
Let $\epsilon>0$. If $E(A) \geq K^{-1}\abs{A}^3$ then there exists $A'\subseteq A$ with $\abs{A'}\gg_\epsilon K^{-100/81-\epsilon}\abs{A}$ such that
\[\abs{A'-A'}\ll_\epsilon K^{100/27+\epsilon}\abs{A'}.\]
\end{theorem}
Note that $100/27\approx 3.703$. The previous best bound, due to Schoen \cite{Sc15}, was $\abs{A'-A'}\ll K^4\abs{A'}$. Here we are chiefly concerned with the size of $A'-A'$, rather than the density of $A'$. Indeed, Reiher and Schoen \cite{ReSc24} have proved a best possible result on the density, showing that there must exist some $A'\subseteq A$ with $\abs{A'}\gg K^{-1/2}\abs{A}$ and $\abs{A'-A'}\ll K^4\abs{A'}$. 

The best possible exponent of $K$ in $\abs{A'-A'}$ would be $c_0=\frac{\log(4)}{\log(27/16)}\approx 2.649$, as shown by Mazur \cite{Ma16}. This exponent is achieved by the intersection of $\bbz^d$ with a sphere (as $d\to \infty$). Indeed, such a family shows that, for any $\epsilon>0$ and any $K$ large enough depending on $\epsilon$, there exist arbitrarily large sets $A$ such that $E(A)\geq K^{-1}\abs{A}^3$ and yet any $X\subseteq A$ with $\abs{X}\geq \abs{A}^{1-\epsilon}$ satisfies $\abs{X-X}\geq K^{c_0-\epsilon}\abs{X}$. In particular the exponent on the doubling constant cannot be improved past $c_0$ even if we weaken the density condition to $\abs{A'}\geq K^{-100}\abs{A}$, say.

In Section~\ref{sec-control} we begin with a general discussion of the notion of control and present some context and history. In Section~\ref{sec-apps} we prove the various applications of control to other additive combinatorial problems discussed above. Finally, in Sections~\ref{sec-energy} and \ref{sec-final} we prove the bounds of Theorem~\ref{th-main}.

\subsection{Notation and conventions}\label{sec-not}

As already mentioned, unless specified otherwise all sets are assumed to be finite subsets of some fixed abelian group $G$, and all functions assumed to be real-valued functions with finite support. We define the convolution of $f,g:G\to \bbr$ by
\[f\ast g(x) = \sum_{y}f(x-y)g(y)\]
and the difference convolution by
\[f\circ g(x)= \sum_y f(x+y)g(y).\]
We use the discrete normalisation, so that the inner product is
\[\langle f, g\rangle = \sum_x f(x)g(x).\]
We note here the useful adjoint property that
\[\langle f\ast g, h\rangle = \langle f, h\circ g\rangle,\]
which is an analytic encoding of the triviality that $x+y=z$ is equivalent to $x=z-y$. For $1\leq p<\infty$ the $L^p$ norm is defined by
\[\norm{f}_p=\brac{\sum_x \abs{f(x)}^p}^{1/p}\]
and $\norm{f}_\infty =\sup_x \lvert f(x)\rvert$. We will make frequent use of the Cauchy-Schwarz inequality and H\"{o}lder's inequality, and in particular note the following consequence of H\"{o}lder's inequality.

\begin{lemma}\label{lem-hol}
For any $f:G\to \bbc$ and $1\leq p<q<r<\infty$
\[\norm{f}_p^{p(r-q)}\leq \norm{f}_q^{q(r-p)}\norm{f}_r^{r(q-p)}.\]
In particular,
\[\norm{f}_{3/2}\leq \norm{f}_3^{1/2}\norm{f}_1^{1/2}.\]
\end{lemma}

The Vinogradov $X \ll Y$ notation means that $X \leq CY$ for some absolute constant $C>0$. If the constant $C$ depends on some other parameter then we use a subscript, such as $X \ll_\epsilon Y$. Similarly, $X\ls_\kappa Y$ denotes inequality up to poly-logarithmic factors. For visual simplicity, we adopt the convention that if the dependendent parameter is $K\geq 1$, as in $X\ls_K Y$, then a loss of $X\leq C_1 (\log K)^{C_2}Y$ is intended, while if the dependent parameter is $\kappa <1$, as in $X\ls_\kappa Y$, then a loss of $\leq C_1 \log(1/\kappa)^{C_2}$ is intended. We omit the subscripts when which parameter is involved in the logarithmic loss is clear from context.

\subsection*{Acknowledgements}
The author is supported by a Royal Society University Research Fellowship. We thank Oliver Roche-Newton, Misha Rudnev, Ilya Shkredov, and Sophie Stevens for their feedback on an early draft of this paper and bringing several papers to our attention.
\section{Control}\label{sec-control}

\subsection{Appearance in the literature}
What we have called control is a rescaled version of the quantity $d^+(A)$ considered by Shakan \cite[Definition 1.5]{Sh19}, which in turn is essentially equivalent to the parameter controlling so-called Szemer\'{e}di-Trotter sets considered by Shkredov \cite[Definition 5]{Sh15}. For comparison to the previous literature, note that $d^+(A)=\kappa \abs{A}$.

To our knowledge, the first place this quantity was explicitly considered (albeit in a rescaled form) was in the work of Konyagin and Shkredov \cite{KoSh15} improving the bounds for the sum-product exponent. 

Shkredov called sets with a non-trivial bound on this $d^+(A)$ Szemer\'{e}di-Trotter sets due to the pivotal role played by the Szemer\'{e}di-Trotter theorem in establishing this type of $L^3$ bound. We will see demonstations of this shortly in the connections to the sum-product problem and convex sets.

The previous bounds known prior to Theorem~\ref{th-main} are summarised below.
\begin{theorem}\label{th-sumset}
If $A$ has control $\kappa$ then
\[E(A) \ls \kappa^{7/13}\abs{A}^3,\]
\[\abs{A-A} \gs \kappa^{-3/5}\abs{A},\]
and
\[\abs{A+A}\gs \kappa^{-11/19}\abs{A}.\]
\end{theorem}

The upper bound on the energy was first proved by Shkredov \cite{Sh13}, although it is only stated there in the special case when $A$ is a convex set (which has control $\ls \abs{A}^{-1}$). Shkredov's proof, which uses the eigenvalue method, immediately generalises to the form given here. 

The lower bound on $\abs{A-A}$ was essentially first proved by Schoen and Shkredov \cite{ScSh11}, although again it is only stated there for convex sets. Their proof easily generalises to the more general notion of control, however, as observed by Murphy, Roche-Newton, and Shkredov \cite[Lemma 16]{MRS17}.

Similarly, the lower bound on $\abs{A+A}$ was essentially proved by Rudnev and Stevens \cite{RuSt22}, although it is only stated there explicitly for convex sets. Again, nothing more about convex sets is used aside from their control properties, and hence their proof naturally generalises to the form stated here. We give a self-contained variant proof of this sumset lower bound in Section~\ref{sec-final}.

Finally, we note that all of the estimates in Theorem~\ref{th-sumset} in fact hold under a slightly weaker assumption than full control: it suffices to assume that $\norm{1_A\circ 1_A}_3\leq \kappa^{1/3}\abs{A}^{4/3}$ and $\norm{1_A\circ 1_B}_2\leq \kappa^{1/4}\abs{A}^{3/4}\abs{B}^{3/4}$ for any set $B$. This bound follows from the control assumption via H\"{o}lder's inequality, but it appears to be weaker (although we know of no meaningful examples to separate this `weak control' from full control, and these would be valuable). In contrast, when proving the new estimates of Theorem~\ref{th-main} we make full use of the $L^3$ control assumption.

The following natural conjecture is natural.
\begin{conjecture}\label{conj1}
If $A$ has control $\kappa$ then, for any $\epsilon>0$,
\[\abs{A-A}\gg_\epsilon \kappa^{-1+\epsilon}\abs{A}\]
and
\[\abs{A+A}\gg_\epsilon \kappa^{-1+\epsilon}\abs{A}.\]
\end{conjecture}
On the other hand, the naive conjecture that $E(A)\ll_\epsilon \kappa^{1-\epsilon}\abs{A}^3$ is too much to hope for -- the best that one can hope for here is
\[E(A)\ls \kappa^{0.7549\cdots}\abs{A}^3.\]
This obstacle arises from the connection to the Balog-Szemer\'{e}di-Gowers theorem, as we will see in Section~\ref{sec-bsg}. It is unclear whether this is the correct exponent, and more constructions of sets with good control and large additive energy would be extremely valuable.

We record here the following convenient lemma, which via a standard `layer-cake' argument shows the control hypothesis yields more a general bound, where $1_B$ can be replaced by arbitrary functions. (Of course the exponent of $100$ in the second summand should not be taken too seriously.)
\begin{lemma}\label{lem-auxcont}
If $A$ has control $\kappa$ then, for any $f:G\to \bbr$, 
\[\norm{1_A\ast f}_3\ls_{\kappa} \kappa^{1/3}\abs{A}^{2/3}\norm{f}_{3/2}+\kappa^{100}\abs{A}\norm{f}_1^{1/3}\norm{f}_\infty^{2/3}.\]
\end{lemma}
\begin{proof}
Without loss of generality, we can assume that $f:G\to [0,1]$. Using the trivial inequalities
\[\norm{1_A\ast f}_3\leq \abs{A}\norm{f}_3\leq \abs{A}\norm{f}_1^{1/3}\norm{f}_{\infty}^{2/3}\]
the contribution from where $0<f(x)\leq \kappa^{200}$ (say) can be bounded above by the second summand on the right, and hence we may assume that $f(x)\geq \kappa^{200}$ on its support.

Let $B_i=\{ x: f(x)\in (2^{-i-1},2^{-i}]\}$. By the triangle inequality
\[\norm{1_A\ast f}_3\leq \sum_i \norm{1_A\ast (f1_{B_i})}_3\ll \sum_i 2^{-i}\norm{1_A\ast 1_{B_i}}_3.\]
Using the control hypothesis we deduce that
\[\norm{1_A\ast f}_3 \ll \kappa^{1/3}\abs{A}^{2/3}\sum_i 2^{-i}\abs{B_i}^{2/3}.\]
Finally, since the sum here is only over $\ll \log(1/\kappa)$ many dyadic scales, 
\[\sum_i 2^{-i}\abs{B_i}^{2/3}\ls_{\kappa} \norm{f}_{3/2}.\]
\end{proof}
We also note the following useful additivity property of control (already observed by Shakan \cite[Lemma 2.4]{Sh19}). 
\begin{lemma}\label{lem-remove}
If $A_1,\ldots,A_t$ are disjoint sets, each with control $\kappa_i$, then $\sqcup A_i$ has control at most $\sum_i \kappa_i$.
\end{lemma}
\begin{proof}
Let $B$ be any finite set and $A'=\sqcup A_i$. By the triangle inequality and H\"{o}lder's inequality
\begin{align*}
\norm{1_{A'}\ast 1_B}_3 
&\leq \sum_i \norm{1_{A_i}\ast 1_{B}}_3\\
&\leq \brac{\sum_i \kappa_i^{1/3}\abs{A_i}^{2/3}}\abs{B}^{2/3}\\
&\leq \brac{\sum_i\kappa_i}^{1/3}\abs{A'}^{2/3}\abs{B}^{2/3}.
\end{align*}
\end{proof}
\section{Applications of control}\label{sec-apps}
\subsection{Application 1: Convex sets}

A finite set of reals $A=\{a_1<a_2<\cdots<a_n\}$ is convex if the sequence of consecutive differences $a_{i+1}-a_i$ is strictly increasing. The classic example is $\{n^2  : 1\leq n\leq N\}$.  

The connection with control is that, for a convex set, an application of the Szemer\'{e}di-Trotter theorem shows that $\kappa \ls \abs{A}^{-1}$. A proof of this fact can be found in a number of places, for example, \cite[Lemma 7]{ScSh11} or \cite[Lemma 2]{RuSt22}. To keep this paper reasonably self-contained, and to illustrate the connection between the Szemer\'{e}di-Trotter incidence theorem and control, we include the proof here. We will use the following incidence bound of Szemer\'{e}di and Trotter \cite{SzTr83}.
\begin{theorem}\label{th-st}
Let $P\subset \bbr^2$ be a finite set of points and $L$ be a finite set of continuous plane curves, such that any two curves in $L$ intersect at most once. The number of incidences between $P$ and $L$ is
\[I(P,L)\ll \abs{P}^{2/3}\abs{L}^{2/3}+\abs{P}+\abs{L}.\]
\end{theorem}
This is all that is required to prove strong control for convex sets. We take this opportunity to prove two other bounds on the control of $f(A)$ where $A$ is arbitrary and $f$ is a convex function, which are useful in other applications. 

\begin{lemma}\label{lem-convcont}
If $A\subset \bbr$ is a finite convex set then $A$ has control $\kappa \ls\abs{A}^{-1}$. 

In general, if $X\subset \bbr$ is any finite set and $f:\bbr\to \mathbb{R}$ is a convex function then, for any finite set $Y$, $A=f(X)$ has control 
\[\ls \frac{\abs{X+Y}^2}{\abs{X}^2\abs{Y}}.\]

Similarly, if $X$ is such that $1_Y\ast 1_Z(x)\geq T$ for all $x\in X$ then $A=f(X)$ has control
\[\ls \frac{1}{T\abs{A}}\brac{\frac{\abs{Y}^2\abs{Z}^2}{T^2\abs{A}}+\abs{Z}+\abs{Y}}.\]
\end{lemma}
\begin{proof}
The first claim follows from the second with $X=Y=\{1,\ldots,N\}$ and $f$ a suitable continuous convex function such that $A=\{ f(n) : 1\leq n\leq N\}$.

To prove the second claim, we let $L$ be the set of curves of the shape
\[(x,f(x))+(y,b)\]
for $y\in Y$ and $b\in B$, so that $\abs{L}=\abs{Y}\abs{B}$. The convexity of $f$ implies that these have the property that any two curves in $L$ intersect at most once. Let $C$ be an arbitrary set and consider the point set $P=(X+Y)\times C$. Each solution to $a+b=c$ gives at least $\abs{Y}$ distinct incidences between $P$ and $L$, since if $a=f(x)$ we can take any $y\in Y$ and the point
\[(x+y, f(x)+b)\in P\]
lies on one of the curves in $L$. It follows by Theorem~\ref{th-st} that
\[\abs{Y}\langle 1_A\ast 1_B,1_C\rangle \leq I(P,L) \ll \abs{X+Y}^{2/3}\abs{Y}^{2/3}\abs{B}^{2/3}\abs{C}^{2/3}+\abs{Y}\abs{B}+\abs{X+Y}\abs{C},\]
whence
\[\langle 1_A\ast 1_B, 1_C\rangle \ll \frac{\abs{X+Y}^{2/3}}{\abs{Y}^{1/3}}\abs{B}^{2/3}\abs{C}^{2/3}+\abs{B}+\frac{\abs{X+Y}}{\abs{Y}}\abs{C}.\]
We now let $C_i=\{ x: 1_A\ast 1_B(x) \in [2^i,2^{i+1})\}$ for $i\geq 0$, so that
\begin{align*}
\sum_x 1_A\ast 1_B(x)^3
&\ll \sum_{i\geq 0} 2^{2i}\langle 1_A\ast 1_B, 1_{C_i}\rangle\\
&\ll \frac{\abs{X+Y}^{2/3}}{\abs{Y}^{1/3}}\abs{B}^{2/3}\sum_{i\geq 0} 2^{2i}\abs{C_i}^{2/3}+\abs{B}\sum_i 2^{2i}+\frac{\abs{X+Y}}{\abs{Y}}\sum_i 2^{2i}\abs{C_i}.
\end{align*}
Since $1_A\ast 1_B(x)\leq \min(\abs{A},\abs{B})$ for all $x$ the sum over $i$ is restricted to where $2^i \ll \abs{A}^{1/2}\abs{B}^{1/2}$, and so the second summand is at most $\abs{A}\abs{B}^2$, and the third summand is at most 
\[\frac{\abs{X+Y}}{\abs{Y}}\sum_x 1_A\ast 1_B(x)^2 \leq \frac{\abs{X+Y}}{\abs{Y}}\abs{A}\abs{B}^2.\]
Finally, the first summand is at most
\[\ls_{\abs{A}}\frac{\abs{X+Y}^{2/3}}{\abs{Y}^{1/3}}\abs{B}^{2/3}\brac{ \sum_x 1_A\ast 1_B(x)^{3}}^{2/3}.\]
Rearranging implies
\[\sum_x 1_A\ast 1_B(x)^3 \ls_{\abs{A}} \brac{\frac{\abs{X+Y}^{2}}{\abs{X}^2\abs{Y}}+\frac{\abs{X+Y}}{\abs{X}\abs{Y}}}\abs{A}^2\abs{B}^2,\]
which yields the required control bound, since trivially $\abs{X+Y}\geq \abs{X}$.

The second bound is proved with a similar argument, this time with $L$ as the set of curves of the shape
\[y\mapsto (y, f(y+z)+b)\]
for $z\in Z$ and $b\in B$ and $P$ as the point set $Y\times C$. Every solution to $a+b=c$ gives rise to $T$ incidences between $P$ and $L$, since if $a=f(x)$ and $x=y+z$ with $y\in Y$ and $z\in Z$ then the point
\[(y, f(y+z)+b)\in P\]
lies on one of the curves in $L$. It follows that
\[T\langle 1_A\ast 1_B, 1_C\rangle \ll \abs{Y}^{2/3}\abs{Z}^{2/3}\abs{B}^{2/3}\abs{C}^{2/3}+\abs{Z}\abs{B}+\abs{Y}\abs{C},\]
whence, arguing as above,
\[\sum 1_A\ast 1_B(x)^3 \ls \brac{\frac{\abs{Y}^2\abs{Z}^2}{T^3\abs{A}^2}+\frac{\abs{Z}}{T\abs{A}}+\frac{\abs{Y}}{T\abs{A}}}\abs{A}^2\abs{B}^2.\]
\end{proof}

Theorem~\ref{th-convex} follows immediately from Lemma~\ref{lem-convcont} and Theorem~\ref{th-main}. Theorem~\ref{th-conv2} requires some preparation: we will prove a useful `decomposition' result that states that, given a convex $f$, any finite set $A\subset \bbr$ can be decomposed into large subsets $X$ and $Y$ such that the control of $f(X)$ and the control of $Y$ is bounded:
\[\kappa_{f(X)}\kappa_Y \ls \abs{A}^{-1}.\]
This is a generalisation of a sum-product decomposition of Shakan \cite[Theorem 1.10]{Sh19}, which corresponds to the special case when $f(x)=\log x$. It is also similar to (although stronger than) \cite[Corollary 5]{StWa22} of Stevens and Warren.

\begin{theorem}\label{th-convdecomp}
Let $A\subset \bbr$ be a finite set and $f:\bbr\to \bbr$ be a continuous convex function. There exist $X,Y\subset A$ such that $X\cup Y=A$ and $\abs{X},\abs{Y}\geq \abs{A}/2$ and $f(X)$ has control $\kappa_{f(X)}$ and $Y$ has control $\kappa_Y$, where
\[\kappa_{f(X)}\kappa_Y\ls_{\abs{A}} \abs{A}^{-1}.\]
\end{theorem}
\begin{proof}
Following the strategy of Shakan \cite{Sh19}, we will first establish that, for any finite $T\subset \bbr$ with control $\kappa_T$, there exists some $T'\subseteq T$ such that $\abs{T'}\gs \kappa_T\abs{T}$ and $f(T')$ has control $\kappa_{f(T')}$, where
\[\kappa_{T}\kappa_{f(T')} \ls \abs{T'}/\abs{T}^2.\]
Indeed, let $B$ be such that
\[\sum 1_{T}\ast 1_B(x)^3 =\kappa_{T} \abs{T}^2\abs{B}^2.\]
By dyadic pigeonholing there exists some $\delta\gg \kappa_T$ and $S$ such that $1_T\ast 1_B(x)\in [\delta,2\delta)\abs{T}$ for all $x\in S$ and $\abs{S}\gs \kappa_T\delta^{-3}\abs{T}^{-1}\abs{B}^2$, so that
\[\langle 1_T,1_S\circ 1_B \rangle = \langle 1_S, 1_T\ast 1_B\rangle \geq \delta\abs{T}\abs{S}.\]
Note that this implies that $\abs{B}\geq \delta \abs{T}$. By another application of dyadic pigeonholing there exists $T'\subseteq T$ and $\eta\gg \delta$ such that $1_S\circ 1_B(x) \in [\eta,2\eta)\abs{S}$ for all $x\in T'$ and $\abs{T'}\gs \eta^{-1}\delta \abs{T}$. By Lemma~\ref{lem-convcont} $f(T')$ has control
\begin{align*}
&\ls (\eta \abs{S}\abs{T'})^{-1}\brac{ \eta^{-2}\abs{T'}^{-1}\abs{B}^2+\abs{B}+\abs{S}}\\
&\ls \kappa_T^{-1}\abs{T'}/\abs{T}^2+\kappa_T^{-1}\delta^2 \abs{B}^{-1}+\delta^{-1}\abs{T}^{-1}.
\end{align*}
Since $\abs{B}\geq \delta \abs{T}$ and $\abs{T'}\gs \delta \abs{T}$ the first term dominates the second. Furthermore, since $\sum_x 1_S\circ 1_B(x)^2 \leq \abs{S}\abs{B}^2$,
\[\eta \delta \abs{T}\abs{S}^2\ls \eta^2\abs{S}^2 \abs{T'}\leq \abs{S}\abs{B}^2\]
whence $\kappa_T\eta \ls \delta^2$, and so the first term dominates the third, which proves the claim.

We may now proceed as Shakan does: we define a sequence of sets $X_i\subseteq A$ by setting $X_0=\emptyset$ and in general construct $X_{i+1}$, provided $X_i\neq A$, by applying the claim with $T=A\backslash X_i$ and setting $X_{i+1}=X_i\cup T'$. Note that since $T'$ is disjoint from $X_i$ and non-empty the size of the $X_i$ must grow until $X_i=A$. Let $i$ be minimal such that $\abs{X_i}\geq \abs{A}/2$, and let $X=X_i$ and $Y=A\backslash X_{i-1}$. By minimality $\abs{Y}\geq \abs{A}/2$. 

We observe that trivially, by definition of control, $\abs{U}^2\kappa_U$ is monotonic, and so for any $1\leq j< i$
\[\abs{Y}^2\kappa_Y\kappa_{f(X_{j+1}\backslash X_j)} \ls \abs{A\backslash X_j}^2\kappa_{A\backslash X_j}\kappa_{f(X_{j+1}\backslash X_j)}\ls \abs{X_{j+1}}-\abs{X_j}.\]
Summing this inequality over $1\leq j<i$ and using Lemma~\ref{lem-remove} we deduce that
\[\abs{Y}^2\kappa_Y\kappa_{f(X)}\ls \abs{X}\]
as required.
\end{proof}
With this decomposition in hand Theorem~\ref{th-conv2} now follows swiftly using the bounds in Theorem~\ref{th-main}.

\begin{proof}[Proof of Theorem~\ref{th-conv2}]
Let $X,Y$ be as in the conclusion of Theorem~\ref{th-convdecomp}. Let $K_1=\abs{A+A}/\abs{A}$ and $K_2=\abs{f(A)+f(A)}/\abs{A}$. Since clearly $\abs{f(X)+f(X)}\ll K_2\abs{X}$ and $\abs{Y+Y}\ll K_1\abs{Y}$, we have, by Theorem~\ref{th-main},
\[(K_1K_2)^{19}\gs (\kappa_{f(X)}\kappa_Y)^{-11}\gs \abs{A}^{11},\]
and hence $K_1K_2 \gs \abs{A}^{11/19}$. The version for difference sets is proved in an identical way. 
\end{proof}

In addition to the consequences recorded in Theorems~\ref{th-convex} and \ref{th-conv2} we note the following conditional result. In particular this highlights that, in some sense, a complete answer to the additive growth of convex sets should follow from control alone.
\begin{theorem}
Assume Conjecture~\ref{conj1}. If $A\subset \bbr$ is a convex set then, for any $\epsilon>0$,
\[\abs{A+A}\gg_\epsilon \abs{A}^{2-\epsilon},\]
and similarly for $\abs{A-A}$. Furthermore, if $A\subset \bbr$ is any finite set and $f:\bbr\to \bbr$ is a continuous convex function then
\[\abs{A+A}\abs{f(A)+f(A)}\gg_\epsilon \abs{A}^{3-\epsilon}\]
and similarly for difference sets.
\end{theorem}

\subsection{Application 2: Sum-product}

For this application, we note that the definition of control only involves a single abelian group operation, and hence makes sense for both addition and multiplication. We will therefore write $\kappa^+$ and $\kappa^\times$ for the additive and multiplicative control of a set.

The link between the sum-product phenomenon and control was first explored by Konyagin and Shkredov \cite{KoSh15}, and has been used in most subsequent papers on the sum-product problem. One particularly clear demonstration of this link was given by Shakan \cite[Theorem 1.10]{Sh19}. As noted in the previous section, this is a special case of our general decomposition result Theorem~\ref{th-convdecomp}. 

\begin{theorem}[Shakan]\label{th-shak}
For any finite $A\subset \bbr$ there exists some $X,Y\subseteq A$ such that $\abs{X}\geq \abs{A}/2$ and $\abs{Y}\geq \abs{A}/2$ and 
\[\kappa^+(X)\kappa^\times(Y)\ls \abs{A}^{-1}.\]
In particular, there exists some $A'\subseteq A$ with $\abs{A'}\gg \abs{A}$ and either additive or multiplicative control $\ls \abs{A}^{-1/2}$. 
\end{theorem}
In particular, Conjecture~\ref{conj1} when coupled with this result of Shakan immediately yields a sum-product exponent of $3/2$.
\begin{corollary}\label{cor-sp}
Assuming Conjecture~\ref{conj1}, for any finite $A\subset \bbr$,
\[\abs{A+A}\abs{AA}\gs \abs{A}^3.\]
In particular
\[\max(\abs{A+A},\abs{AA})\gs \abs{A}^{3/2}.\]
\end{corollary}
It may be of interest to compare Corollary~\ref{cor-sp} to the bound of Solymosi \cite{So09}, that
\[\abs{A+A}^2\abs{AA} \gs \abs{A}^4.\]
Perhaps the easiest link between sum-product problems and control arises from the Szemer\'{e}di-Trotter theorem, as noted by Konyagin and Shkredov \cite{KoSh15}, which implies that if $A$ has small product set then it must have good additive control. While we will not use this particular relationship in this paper, we include it for comparison.
\begin{lemma}\label{lem-sp1}
If $A\subset \bbr$ is a finite set with $\abs{AA}\leq K\abs{A}$ (or $\abs{A/A}\leq K\abs{A}$) then $A$ has (additive) control
\[\ll K^2\abs{A}^{-1}.\]
\end{lemma}
\begin{proof}
This is a single application of the Szemer\'{e}di-Trotter theorem, similar to the proof of Lemma~\ref{lem-convcont}, and can be found in a number of papers -- for example, one can apply Lemma 2 of \cite{RuSt22} with $\Pi_1=A^{-1}$ and $\Pi_2=AA$ and $T=\abs{A}$.
\end{proof}

Rudnev and Stevens \cite{RuSt22}, building on ideas of Konyagin and Shkredov \cite{KoSh15}, gave a more intricate argument, which in most cases is quantitatively superior to Lemma~\ref{lem-sp1}. We give a variant of this argument below.

\begin{lemma}\label{lem-piece}
If $A\subset\bbr$ is a finite set with $\max(\abs{A+A},\abs{AA})\leq K\abs{A}$ then there is a subset $A'\subseteq A$ 
such that $A'$ has (additive) control
\[\ls \abs{A'}K^{83/2}\abs{A}^{-31/2}.\]
\end{lemma}
\begin{proof}
By the Cauchy-Schwarz inequality we have $E^\times(A) \geq \abs{A}^3/K$. Let $r_{A/A}(x)$ count the number of representations of $x$ as $x=a/b$ with $a,b\in A$. By the dyadic pigeonhole principle there exists some $\delta>0$ such that $r_{A/A}(x)\in[\delta,2\delta)\abs{A}$ for all $x\in S$, say, and $\abs{S}\gs \delta^{-2}K^{-1}\abs{A}$. By the argument in the proof of Theorem 1 of \cite{RuSt22} there exists some $A'\subseteq A$ such that 
\[\abs{A'}\gg \delta\abs{S}\]
such that, for every $x\in A'$,
\[r_{AA/AA}(x) \gs K^{-12}\abs{S}^{-1/2}\abs{A}^6,\]
where $r_{AA/AA}(x)$ counts the number of representations of $x=a/b$ with $a,b\in AA$. It follows by Lemma 2 of \cite{RuSt22} that, for any set $B$,
\begin{align*}
\sum 1_{A'}\circ 1_B(x)^3 
&\ls K^4\abs{A}^4\abs{B}^2\brac{K^{-12}\abs{S}^{-1/2}\abs{A}^6}^{-3}\\
&\ls \brac{\delta^{-3}\abs{S}^{-3/2}K^{40}\abs{A}^{-14}}\abs{A'}^3\abs{B}^2\\
&\ls \brac{K^{83/2}\abs{A}^{-31/2}}\abs{A'}^3\abs{B}^2,
\end{align*}
where in the final inequality we have used the fact that $\abs{S}\gs \delta^{-2}K^{-1}\abs{A}$. 
\end{proof}

In \cite{RuSt22} Lemma~\ref{lem-piece} was (implicitly) invoked, along with the lower bound $\abs{A'}\gg K^{-1}\abs{A}$. Here we note that, using Lemma~\ref{lem-remove}, the previous lemma can be improved via iteration, essentially allowing one to assume $\abs{A'}\gg \abs{A}$. This observation alone yields our improved sum-product exponent (in particular does not use the improvements of Theorem~\ref{th-main}). 

\begin{lemma}\label{lem-union}
If $A\subset \bbr$ is a finite set with $\max(\abs{A+A},\abs{AA})\leq K\abs{A}$ then there is a subset $A'\subseteq A$ with $\abs{A'}\geq \abs{A}/2$ such that $A'$ has (additive) control
\[\ls K^{83/2}\abs{A}^{-29/2}.\]
\end{lemma}
\begin{proof}
We repeatedly apply Lemma~\ref{lem-piece}, removing $\cup A_i$ from $A$ each time, until the union has size at least $\abs{A}/2$.
\end{proof}

With Lemma~\ref{lem-union} in hand it is now a simple matter to convert results on control (as in Theorem~\ref{th-main}) into improved bounds for the sum-product exponent.
\begin{theorem}
If $a>0$ is such that any set $A$ with control $\kappa$ has $\abs{A+A}\gg \kappa^{-a}\abs{A}$ then, for any finite $A\subset \bbr$,
\[\max(\abs{A+A},\abs{AA})\gg \abs{A}^{1+c}\]
where
\[c=\max\brac{\frac{a}{2},\frac{29a}{83a+2}}.\]
Similarly, if any set $A$ with control $\kappa$ has $\abs{A-A}\gg \kappa^{-a}\abs{A}$ then, for any finite $A\subset \bbr$,
\[\max\abs{A-A},\abs{A+A},\abs{AA})\gg \abs{A}^{1+c}.\]
\end{theorem}
We note that the former bound for $c$ (which arises from the Shakan decomposition) is superior to the second bound (which arises from the Rudnev-Stevens modifiation of the Konyagin-Shkredov argument) once $a\geq 56/83\approx 0.675$.

Theorem~\ref{th-spmain} follows immediately via Theorem~\ref{th-main}, taking $a=11/19-\epsilon$ and $a=2506/4175-\epsilon$ respectively.
\begin{proof}
For the bound $c\geq a/2$ we apply Shakan's decomposition result Theorem~\ref{th-shak}. On the other hand, if $\max(\abs{A+A},\abs{AA})=K\abs{A}$ then by Lemma~\ref{lem-union} we have
\[\abs{A+A}\gs \brac{K^{83/2}\abs{A}^{-29/2}}^{-a}\abs{A},\]
and rearranging yields the result. The bound involving $\abs{A-A}$ is proved in an identical fashion.
\end{proof}

To prove Theorem~\ref{th-spenergy} we simply need to repeat the argument of Xue \cite{Xu21} with the new bounds of Theorem~\ref{th-main}; we sketch the details below.

\begin{proof}[Proof of Theorem~\ref{th-spenergy}]
By \cite[Theorem 1.10]{Xu21} there exist $W,Z\subseteq A$ with $\abs{W},\abs{Z}\gg \abs{A}$ such that
\[\norm{1_W\circ 1_W}_3^{12}E_\times(Z)^3 \ls \abs{A}^{22}.\]
Arguing as in the proof of \cite[Lemma 4.4]{Xu21}, there exist $X,Y\subseteq W$ such that $\abs{X},\abs{Y}\gg \abs{A}$ and $X$ has additive control $\kappa$ and $Y$ has multiplicative control $\lambda$, where $\lambda \kappa \ls \abs{A}^{-1}$ and
\[E_+(X)^{13/2}\ls \kappa^{1/2}\abs{A}^{15/2}\norm{1_W\circ 1_W}_3^9,\]
whence
\[E_+(X)^{13/2}\ls \lambda^{-1/2}\abs{A}^{7}\norm{1_W\circ 1_W}_3^{9}.\]
Since $E_\times(Y)\ll_\epsilon \lambda^{27/50-\epsilon}\abs{A}^3$ we have
\[\lambda^{-1/2}\ll_\epsilon E_\times(Y)^{-25/27+\epsilon}\abs{A}^{25/9+\epsilon}\]
whence
\[E_+(X)^{13/2}E_\times(Y)^{25/27-\epsilon}\ll_\epsilon \abs{A}^{88/9+\epsilon}\norm{1_W\circ 1_W}_3^{9}.\]
It follows that 
\[E_+(X)^{13/2}E_\times(Y)^{25/27-\epsilon}E_\times(Z)^{9/4}\ll_\epsilon \abs{A}^{473/18+\epsilon},\]
and the claim follows.
\end{proof}

Similarly, the improved estimates of Theorem~\ref{th-spshift} require combining the arguments of Murphy, Roche-Newton, and Shkredov \cite{MRS17} with Theorem~\ref{th-main}. 
\begin{proof}[Proof of Theorem~\ref{th-spshift}]
Adapting the proof of \cite[Lemma 11]{MRS17} to use the improved energy bound of Theorem~\ref{th-main}, for any $A\subset \bbr$ (with control $\kappa$) and $\alpha \in \bbr\backslash\{0\}$ we have
\[\abs{A(A+\alpha)}\gg_\epsilon \kappa^{-27/50+\epsilon}E_\times(A)^2\abs{A}^{-5}.\]
Inserting this into the proof of \cite[Lemma 14]{MRS17} in turn implies that, if $E_\times(A)=\tau \abs{A}^3$, then for any $\alpha \in \bbr\backslash \{0\}$ we have
\[\abs{A(A+\alpha)}\gg_\epsilon \tau^{127/50+\epsilon}\abs{A}^{77/50}.\]
On the other hand, by \cite[Corollary 9]{MRS17} there exists some $\alpha\in A$ such that
\[\abs{A(A+\alpha)}\gg_\epsilon \tau^{-1/2+\epsilon}\abs{A}^{3/2}.\]
The required lower bound for $\abs{A(A+\alpha)}$ then follows by balancing these two inequalities.

For the next estimate, by \cite[Lemma 18]{MRS17} if $E_\times(A)= \tau\abs{A}^3$ there exists $A'\subseteq A$ and $\delta \gg \tau$ such that $\abs{A'}\gs \delta^{-1} \tau\abs{A}$ and $A'$ has multiplicative control $\ll \delta^{-1}\tau^{-1}\abs{A'}/\abs{A}^2$. It follows that
\[\abs{A(A-A)}\geq \abs{A'-A'}\gg_\epsilon (\delta^{-1}\tau^{-1}\abs{A'}/\abs{A}^2)^{-2506/4175+\epsilon}\abs{A'},\]
whence
\[\abs{A(A-A)}\gg_\epsilon \tau^{5012/4175+\epsilon}\abs{A}^{6681/4175-\epsilon}.\]
On the other hand, \cite[Corollary 9]{MRS17} implies
\[\abs{A(A-A)}\gs \tau^{-1/2}\abs{A}^{3/2}.\]
Balancing these two inequalities yields
\[\abs{A(A-A)}\gg_\epsilon \abs{A}^{7239/4733-\epsilon}. \]

Similarly, we have
\[\abs{A(A+A)}\gs \max(\tau^{-1/2}\abs{A}^{3/2},\tau^{22/19}\abs{A}^{30/19}),\]
and so
\[\abs{A(A+A)}\gg \abs{A}^{32/21}.\]
\end{proof}

Finally, we explain how to prove Theorem~\ref{th-pachzakh}, the improved bounds on bi-Sidon sets. As with the other results, this is a simple matter of inserting our new quantitative bounds into the proof of Pach and Zakharov \cite{PaZa24}. 

\begin{proof}[Proof of Theorem~\ref{th-pachzakh}]
By Theorem~\ref{th-spenergy} there exists $A'\subseteq A$ with $\abs{A'}\gg \abs{A}$ such that, for any $\epsilon>0$,
\[\min(E_\times(A'),E_+(A')) \ll_\epsilon \abs{A}^{3-27/95+\epsilon}.\]
We then proceed as in the proof of \cite[Theorem 1]{PaZa24}, using this estimate in place of \cite[Theorem 3]{PaZa24}. 
\end{proof}
\subsection{Application 3: New bounds for the Balog-Szemer\'{e}di-Gowers theorem}\label{sec-bsg}

The Balog-Szemer\'{e}di-Gowers theorem, part of the standard toolkit of additive combinatorics, says that any set with large additive energy must contain a large subset with small sumset. (The converse direction, that a small sumset implies a large additive energy, is an immediate consequence of the Cauchy-Schwarz inequality.) This was first proved qualitatively by Balog and Szemer\'{e}di \cite{BaSz94}, but reproved with polynomial bounds by Gowers \cite{Go01}.

There is a strong connection between $\kappa$ and the bounds available for the Balog-Szemer\'{e}di-Gowers theorem. In particular we have the following.

\begin{theorem}\label{th-bsg}
If there exists $B$ such that $\norm{1_A\ast 1_B}_3\geq \kappa^{1/3}\abs{A}^{2/3}\abs{B}^{2/3}$ then then there exists $A'\subseteq A$ with $\abs{A'}\gs_\kappa \kappa^{2/3}\abs{A}$ such that
\[\abs{A'-A'}\ls_\kappa \kappa^{-2}\abs{A'}.\]
\end{theorem}

\begin{corollary}\label{cor-bsg1}
If $A$ having control $\kappa$ implies $E(A) \ll \kappa^c\abs{A}^3$ then the following holds. If $E(A) \geq K^{-1}\abs{A}^3$ then there exists $A'\subseteq A$ with $\abs{A'}\gg K^{-2/3c}\abs{A}$ such that
\[\abs{A'-A'}\ll K^{2/c}\abs{A'}.\] 
\end{corollary}
Theorem~\ref{th-bsgmain} follows immediately using Theorem~\ref{th-main}. Note that Corollary~\ref{cor-bsg} implies that, when combined with the sphere construction of Mazur \cite{Ma16} discussed previously, in any inequality of the shape 
\[E(A) \ll \kappa^c \abs{A}^3\]
(under the assumption that $A$ has control $\kappa$) one could not do better than 
\[c = \frac{\log(27/16)}{\log 2}\approx 0.7549.\]

Note that using the trivial bound $E(A) \leq \kappa^{1/2}\abs{A}^3$ already yields the fact that if $E(A) \geq K^{-1}\abs{A}^3$ then there is $A'\subseteq A$ with $\abs{A'-A'}\ll K^4\abs{A'}$, which was the previous best exponent known, proved by Schoen \cite{Sc15}. Indeed, our proof of Theorem~\ref{th-bsg} is along similar lines to the argument of Schoen, particularly that of \cite[Theorem 1.2]{Sc15}, the novelty being that we allow for an arbitrary set $B$ in exploiting a lower bound for $\norm{1_A\ast 1_B}_3$, rather than fixing $B=A$.

In the remainder of this section we prove Theorem~\ref{th-bsg}. We begin with a lemma that is a variant of \cite[Lemma 2.1]{Sc15}.
\begin{lemma}\label{lem-bsg}
If $A,B,C$ are finite sets such that $1_A\ast 1_B(x)\geq \eta\abs{A}$ for all $x\in C$ then there exists some $X\subseteq A$ such that $\abs{X}\geq \eta\abs{A}$ and, for all but at most $\epsilon\abs{X}^2$ many pairs $(a,b)\in X^2$, 
\[\ind{A}\circ\ind{A}(a-b)> \epsilon \eta^2\frac{\abs{A}\abs{C}}{\abs{B}^2}\abs{A}.\] 
\end{lemma}
\begin{proof}
If $G\subseteq B^2$ is the set of $(a,b)$ such that $1_A\circ 1_A(a-b) \leq \epsilon\eta^2\abs{A}^2\abs{C}/\abs{B}^2$ then (using the trivial bound $\abs{G}\leq \abs{B}^2$)
\[\sum_{(a,b)\in G}\sum_{x\in C}1_A(x-a)1_A(x-b)\leq \epsilon \eta^2\abs{A}^2\abs{C}.\]
If we let $X_x=B\cap (x-A)$ then $\abs{X_x}=1_A\ast 1_B(x)$ and so
\[\sum_{x\in C}\Abs{X_x^2\cap G}=\sum_{(a,b)\in G}\sum_{x\in C}1_A(x-a)1_A(x-b)\leq \epsilon\eta^2\abs{A}^2\abs{C}\leq \epsilon\sum_{x\in C} \abs{X_x}^2.\]
In particular there exists some $X=X_x$ such that $\abs{X}\geq \eta\abs{A}$ and at most $\epsilon\abs{X}^2$ many pairs $(a,b)\in X^2$ are in $G$, which is the conclusion (after translating $X$ to be a subset of $A$).
\end{proof}

\begin{corollary}\label{cor-bsg}

If $A,B,C$ are such that $1_A\ast 1_B(x)\geq \eta\abs{A}$ for all $x\in C$ then there exists some $A'\subseteq A$ of size $\gg \eta\abs{A}$ such that
\[\abs{A'-A'} \ll \eta^{-6}\frac{\abs{B}^4}{\abs{A}^2\abs{C}^2}\abs{A'}.\]
\end{corollary}
\begin{proof}
We let $X$ and $G$ be the set and graph provided by Lemma~\ref{lem-bsg} with $\epsilon=1/8$, and let $A'\subseteq X$ be those elements with degree at least $\frac{3}{4}\abs{X}$ in $G$. We know that $G$ has at least $\frac{7}{8}\abs{X}^2$ many edges, and hence $\abs{A'}\geq 2^{-4}\eta\abs{A}$. Furthermore, for any $a-b\in A'-A'$ there exist $\gg \abs{X}$ many $c\in X$ such that $(a,c)$ and $(b,c)$ are edges in $G$. This means that both $a-c$ and $b-c$ have $\gg \eta^2\abs{A}^2\abs{C}\abs{B}^{-2}$ representations as $x-y$ where $x,y\in A$. Using the identity
\[a-b=(a-c)-(b-c)\]
we deduce that each element of $A'-A'$ has $\gg \eta^5\abs{A}^5\abs{C}^2\abs{B}^{-4}$ representations $x_1+x_2-x_3-x_4$ where $x_i\in A$, and hence
\[\abs{A'-A'}\eta^5\abs{A}^5\abs{C}^2\abs{B}^{-4} \ll \abs{A}^4\]
whence
\[\abs{A'-A'}\ll \eta^{-5}\abs{A}^{-1}\abs{B}^4\abs{C}^{-2}\ll \eta^{-6}\frac{\abs{B}^4}{\abs{A}^2\abs{C}^2}\abs{A'}.\]
\end{proof}

We may now prove Theorem~\ref{th-bsg} as a consequence of Corollary~\ref{cor-bsg}.
\begin{proof}[Proof of Theorem~\ref{th-bsg}]
Let $B$ be such that 
\[\sum_x 1_A\ast 1_B(x)^3=\kappa \abs{A}^2\abs{B}^2.\]
By dyadic pigeonholing there exists some $\eta\gg\kappa$ and $C$ such that $1_A\ast 1_B(x)\geq \eta \abs{A}$ for all $x\in C$ and $\abs{C}\gs_\kappa \kappa \eta^{-3}\abs{A}^{-1}\abs{B}^2$. On one hand, applying Corollary~\ref{cor-bsg} immediately yields some $A'\subseteq A$ of size $\abs{A'}\gg \eta \abs{A}$ such that
\[\abs{A'-A'}\ll \eta^{-6}\frac{\abs{B}^4}{\abs{A}^2\abs{C}^2}\abs{A'}\ls \kappa^{-2}\abs{A'}.\]
On the other, we know that by definition
\[\langle 1_B, 1_C\circ 1_A\rangle=\langle 1_A\ast 1_B, 1_C\rangle\geq \eta\abs{A}\abs{C},\]
and hence by dyadic pigeonholing there exists some $\rho\gg \eta\abs{C}/\abs{B}$ such that $1_A\circ 1_C(x)\geq \rho\abs{A}$ for all $x\in B'$, where $B'\subseteq -B$ has size $\abs{B'}\gs \eta\rho^{-1}\abs{C}$.

By another application of Corollary~\ref{cor-bsg} we find some $A'\subseteq A$ of size 
\[\abs{A'}\gg \rho\abs{A}\gg \eta\abs{C}\abs{A}/\abs{B}\gs \kappa \eta^{-2}\abs{B}\]
such that
\[\abs{A'-A'}\ll \rho^{-6}\frac{\abs{C}^4}{\abs{A}^2\abs{B'}^2}\abs{A'}\ls  \eta^{-6}\frac{\abs{B}^4}{\abs{A}^2\abs{C}^2}\abs{A'}\ls \kappa^{-2}\abs{A'}.\]
In either case, we have found some $A'\subseteq A$ such that $\abs{A'-A'}\ls \kappa^{-2}\abs{A'}$ and
\[\abs{A'}\gs \max\brac{ \eta\abs{A}, \kappa \eta^{-2}\abs{B}}.\]
We note that the trivial bound 
\[\sum 1_A\ast 1_B(x)^3 \leq \abs{A}\abs{B}^3\]
implies $\abs{B}\geq \kappa \abs{A}$, and hence the maximum above is at least $\kappa^{2/3}\abs{A}$ as claimed.
\end{proof}

This concludes our discussion of the applications of control. The remainder of this paper will focus on proving Theorem~\ref{th-main}, deducing good quantitative bounds on other measures of additive structure from a control hypothesis.
\section{Applications of bounding $\norm{1_A\circ 1_S}_{3/2}$}\label{sec-energy}
All of the threshold-breaking bounds on $\abs{A+A}$, $\abs{A-A}$, and $E(A)$ (both in this paper and the previous literature) arise from using the control hypothesis to relate these quantities to $\norm{1_A\circ 1_S}_{3/2}$, where $S$ is some `symmetry set'; that is, a set of the form $\{ x: 1_A\circ 1_A(x)\geq \delta\abs{A}\}$ for some $\delta>0$.

In this section we make these connections explicit. For reference, it may be useful to note that by H\"{o}lder's inequality and the control hypothesis
\begin{equation}\label{eq-holderbound}
\norm{1_A\circ 1_S}_{3/2}\leq \abs{A}^{1/2}\abs{S}^{1/2}\norm{1_A\circ 1_S}_{3}^{1/2}\leq \kappa^{1/6}\abs{A}^{5/6}\abs{S}^{5/6}.
\end{equation}
The threshold-breaking bounds in the previous literature all arise from using this `trivial' bound combined with (something equivalent to) the arguments in this section. In the following section we will prove improved bounds for $\norm{1_A\circ 1_S}_{3/2}$ in certain regimes, which is the source of our improvements for $\abs{A-A}$ and $E(A)$. We have not been able to find a bound which offers an improvement in the range relevant for improving the bounds for $\abs{A+A}$.

We first present the simplest application, which is to bound the size of $A-A$. This is essentially the argument of Schoen and Shkredov \cite{ScSh11}, and indeed one can check that using \eqref{eq-holderbound} in conjunction with the following lemma yields the Schoen-Shkredov bound of $\abs{A-A}\gg \kappa^{-3/5}\abs{A}$.
\begin{lemma}\label{lem-diffapp}
If $A$ has control $\kappa$ and $\abs{A-A}=K\abs{A}$ then
\[\abs{A}^{5/3} \ll \kappa^{4/3}K^{5/3}\norm{1_A\circ 1_S}_{3/2},\]
where
\[S = \{ x : 1_A\circ 1_A(x) \geq (2K)^{-1}\abs{A}\}.\]
\end{lemma}
We remark that a similar argument proves the bound
\[\norm{1_A\circ 1_A}_{5/4}^5\ls_K \kappa^{4/3}K^{2/3}\abs{A}^{22/3}\norm{1_A\circ 1_S}_{3/2},\]
which implies the previous inequality (up to logarithmic factors) via H\"{o}lder's inequality. This stronger form does not seem to produce any quantitative improvement, however.
\begin{proof}
Let $S=\{x : 1_A\circ 1_A(x)\geq (2K)^{-1}\abs{A}\}$, so that
\[\abs{A}^2 \ll \langle 1_A\circ 1_A, 1_S\rangle=\langle 1_A, 1_A\ast 1_S\rangle.\]

By the Cauchy-Schwarz inequality
\[\abs{A}^3 \ll \langle 1_A, (1_A\ast 1_S)^2\rangle = \sum_{x,y\in S}\sum_{a\in A}1_A(a-x)1_A(a-y).\]
The innermost sum ensures that this sum is restricted to those $x,y\in S$ such that $x-y\in A-A$, and hence by the Cauchy-Schwarz inequality
\[\abs{A}^6 \ll \langle 1_S\circ 1_S, 1_{A-A}\rangle\sum_{x,y\in S}\brac{\sum_{a\in A}1_A(a-x)1_A(a-y)}^2.\]
We have
\begin{align*}
\sum_{x,y\in S}\brac{\sum_{a\in A}1_A(a-x)1_A(a-y)}^2
&=\sum_{a_1,a_2\in A}\brac{\sum_{x\in S}1_A(a_1-x)1_A(a_2-x)}^2\\
&\leq \sum_z 1_A\circ 1_A(z)^3\\
&\leq \kappa\abs{A}^4,
\end{align*}
and hence 
\[\abs{A}^2 \ll \kappa \langle 1_S\circ 1_S, 1_{A-A}\rangle\ll \kappa K\abs{A}^{-1}\langle 1_A\circ 1_A\circ 1_S, 1_{A-A}\rangle.\]
By H\"{o}lder's inequality
\[\langle 1_A\circ 1_A\circ 1_S, 1_{A-A}\rangle=\langle 1_A\ast 1_{A-A}, 1_A\circ 1_S\rangle \leq \norm{1_A\ast 1_{A-A}}_3\norm{1_A\circ 1_S}_{3/2},\]
and the conclusion follows from $ \norm{1_A\circ 1_{A-A}}_3\leq \kappa^{1/3}K^{2/3}\abs{A}^{4/3}$.
\end{proof}

It is more challenging to bound $\abs{A+A}$, essentially because the use of the Cauchy-Schwarz inequality naturally creates sums with restrictions to the symmetric set $A-A$. To produce $A+A$ instead requires some trickery. The following argument is a variant of that of Rudnev and Stevens \cite{RuSt22}. Again, using this lemma in conjunction with \eqref{eq-holderbound} (and the bounds $\tau \geq K^{-1}$ and $\abs{S}\leq \delta^{-2}\tau\abs{A}$) produces the Rudnev-Stevens  bound $\abs{A+A}\gs_\kappa \kappa^{-11/19}\abs{A}$. Unlike the other two applications in this section, we have not found a bound for $\norm{1_A\circ 1_S}_{3/2}$ which improves upon \eqref{eq-holderbound} in the range relevant for this application.
\begin{lemma}\label{lem-sumapp}
Suppose $A$ has control $\kappa$ and $E(A)=\tau\abs{A}^3$. If $\abs{A+A}=K\abs{A}$ (where $K\leq \kappa^{-1}$) then there exists some $K^{-1}\ll\delta\ls_\kappa \kappa\tau^{-1}$ such that
\[\tau^{2}\abs{A}^{5/3} \ls_{\kappa} \kappa^{4/3}K^{5/3}\delta^{2}\norm{1_A\circ 1_S}_{3/2}.\]
where
\[S = \{ x : 1_{A}\circ 1_A(x) \geq \delta\abs{A}\}.\]
\end{lemma}
\begin{proof}
(All logarithmic losses in this proof are relative to $\kappa$.) We note that $E(A)=\norm{1_A\circ 1_A}_2^2=\norm{1_A\ast 1_A}_2^2 \geq \abs{A}^3/K$. It follows by dyadic pigeonholing that there exists some $\delta \gg \tau\geq K^{-1}$ such that if 
\[S=\{ x: 1_A\circ 1_A(x)\geq \delta \abs{A}\}\]
then $\abs{S}\gs \delta^{-2}\tau\abs{A}$. The fact that $\sum 1_A\circ 1_A(x)^3 \leq \kappa \abs{A}^4$ immediately implies $\delta \ls \kappa\tau^{-1}$. By definition
\[\langle 1_A, 1_A\ast 1_S\rangle=\langle 1_S, 1_A\circ 1_A\rangle\geq \delta\abs{A}\abs{S}.\]
It follows that there exists some $A_1\subseteq A$ such that $1_A\ast 1_S(x)\gg \delta\abs{S}\gs \delta^{-1}\tau\abs{A}$ for all $x\in A'$.

If $\abs{A_1}\geq \abs{A}/2$ we proceed with the remainder of the proof. Otherwise, we repeat the preceding argument with $A\backslash A_1$, noting that we still have $E(A\backslash A_1)\gg \abs{A}^3/K$. Repeating this, we produce disjoint $A_1,\ldots,A_t$ with associated $\delta_i$ and $S_i$, such that $1_A\ast 1_{S_i}(x)\gs \delta_i^{-1}\tau\abs{A}$ for all $x\in A_i$, where $\abs{\cup_i A_i}\gg \abs{A}$. By a further dyadic pigeonholing we can find some sub-collection of the $A_i$ in which the corresponding $\delta_i$ are all in the same dyadic range, and whose union has size $\gs \abs{A}$. If we then let $A'$ be the union of all such $A_i$ then we have found some $A'\subseteq A$ such that
\begin{enumerate}
\item $\abs{A'}\gs \abs{A}$ and
\item there exists $\delta\gg \tau$ such that if
\[S= \{ x: 1_A\circ 1_A(x)\geq \delta\abs{A}\}\]
then for all $x\in A'$ we have $1_A\ast 1_S(x)\gs \delta^{-1}\tau\abs{A}$.
\end{enumerate}

We now let $T=\{ x: 1_{A'} \ast 1_A(x)\geq (2K)^{-1}\abs{A}\}$, so that
\[\langle 1_{A'}, 1_T\circ 1_A\rangle=\langle 1_{A'}\ast 1_A, 1_T\rangle \gg \abs{A}^2.\]
It follows that
\[\delta^{-1}\tau\abs{A}^3 \ls \langle 1_{A'},(1_T\circ 1_A)(1_A\ast 1_S)\rangle= \sum_{x\in T}\sum_{y\in S}\sum_{a\in A'}1_A(x-a)1_A(a-y).\]
We can now proceed as in the proof of Lemma~\ref{lem-diffapp}, except that now the sum is restricted to those pairs $x,y$ such that $x-y\in A+A$, and hence
\[\delta^{-2}\tau^2\abs{A}^2\ls \kappa \langle 1_T\circ 1_S, 1_{A+A}\rangle.\]
We have
\[\langle 1_T\circ 1_S, 1_{A+A}\rangle\ll K\abs{A}^{-1}\langle 1_A\circ 1_{A+A},1_S\circ 1_A\rangle\ll \kappa^{1/3}K^{5/3}\abs{A}^{1/3}\norm{1_A\circ 1_S}_{3/2},\]
and the claim follows.
\end{proof}

Finally, we explain the link between $\norm{1_A\circ 1_S}_{3/2}$ and the additive energy $E(A)$. In the work of Shkredov \cite{Sh13} this is accomplished via a spectral argument, but here we use only elementary methods. We require the following auxiliary lemma.

\begin{lemma}\label{lem-innbound}
For any sets $A,B,S$
\[\langle 1_{A}\ast 1_{B}, 1_S\rangle^8 \leq \abs{A}^2\abs{B}^4\abs{S}^2\norm{1_A\circ 1_A}_3^3\norm{1_B\circ 1_B}_3^2\norm{1_S\circ 1_S}_3.\]
\end{lemma}
\begin{proof}
We have, by the Cauchy-Schwarz inequality,
\[\langle 1_{A}\ast 1_{B}, 1_S\rangle^2 =\langle 1_B, 1_S\circ 1_A\rangle^2\leq \abs{B}\sum_{\substack{a\in A\\ x\in S}}\sum_{y\in S}1_{B}(y-a)1_{A}(x-y+a).\]
By the Cauchy-Schwarz inequality once again,
\[\langle 1_{A}\ast 1_{B}, 1_S\rangle^4 \leq \abs{B}^2\abs{A}\abs{S}\sum_{y_1,y_2\in S}\sum_{a\in A}1_B(y_1-a)1_A(y_2-a)\sum_{x\in S}1_{A}(x-y_1+a)1_A(x-y_2+a).\]
The innermost sum is at most $1_{A}\circ 1_{A}(y_1-y_2)$ and
\[\sum_{y_1,y_2\in S}\brac{\sum_{a\in A}1_B(y_1-a)1_B(y_2-a)}^2\leq \langle 1_A\circ 1_A,(1_B\circ 1_B)^2\rangle,\]
hence
\[\langle 1_{A}\ast 1_{B}, 1_S\rangle^8 \leq \abs{A}^2\abs{B}^4\abs{S}^2\langle 1_A\circ 1_A,(1_B\circ 1_B)^2\rangle \langle 1_S\circ 1_S,(1_A\circ 1_A)^2\rangle.\]
The conclusion now follows from H\"{o}lder's inequality.
\end{proof}

Unfortunately, our method for bounding $\norm{1_A\circ 1_S}_{3/2}$ in the range relevant for $E(A)$ requires several additional technical assumptions. These can be met, which is the content of the following lemma, but it does make the connection harder to parse. On a first reading, the reader should take $A_1=A_2=A$ and $S=S'$ a $\delta$-level symmetry set of $1_A\circ 1_A$ as before. With this simplification, coupled with \eqref{eq-holderbound} and the bound $\abs{S}\leq \delta^{-2}\tau\abs{A}$, the inequality in the following lemma yields the Shkredov \cite{Sh13} energy bound of $E(A)\ls \kappa^{7/13}\abs{A}^3$.

\begin{lemma}\label{lem-enapp}
Let $A$ have control $\kappa$ and $E(A)=\tau \abs{A}^3$. There exists $\tau\ll \delta \ls \kappa\tau^{-1}$ and $A_1,A_2\subseteq A$ and $S,S'$ such that, for any $k\geq 1$
\begin{enumerate}
\item $\delta^{-2}\tau \abs{A}\ls \abs{S'}\ls \delta^{-2}\tau\abs{A}$,
\item $1_{A_2}\circ 1_A(x)\gg 2^{-O(k)}\delta\abs{A}$ for all $x\in S$,
\item $1_{A_1}\circ 1_{S'}(x)\gg 2^{-O(k)}\delta\abs{S}$ for all $x\in A_2$, and
\item $\norm{1_{A_1}\circ 1_{S}}_{3/2}\leq (\tau\kappa)^{-O(1/k)}\norm{1_{A_2}\circ 1_{S}}_{3/2}$,
\end{enumerate}
and
\[\tau^6\abs{A}^{5/3}\ll 2^{O(k)}\kappa^2\delta^3\norm{1_{A_2}\circ 1_S}_{3/2},\]
\end{lemma}

\begin{proof}
By dyadic pigeonholing there exists some $\delta \gg \tau$ such that, if
\[S'=\{ x: 1_A\circ 1_A(x)\geq \delta\abs{A}\},\]
then $\abs{S'}\gs \delta^{-2}\tau\abs{A}$ (which immediately implies $\delta \ls \kappa\tau^{-1}$). Note that $S'$ is symmetric (so that $f\circ 1_{S'}$ is the same as $f\ast 1_{S'}$).

We claim that there exists a sequence of sets $A_i\subseteq A$ for all integer $i\geq 0$ such that
\begin{enumerate}
\item $\langle 1_{S'}, 1_{A_i}\circ 1_{A_{i+1}}\rangle \geq 2^{-i-1}\delta\abs{A}\abs{S'}$ and
\item $1_{S'}\ast 1_{A_i}(x)\geq 2^{-i-1}\delta\abs{S'}$ for all $x\in A_{i+1}$.
\end{enumerate}
The proof is by induction: we begin by noting that
\[\langle 1_A, 1_{S'}\ast 1_A\rangle=\langle 1_{S'}, 1_A\circ 1_A\rangle \geq \delta\abs{A}\abs{S'}.\]
We then set $A_0=A$ and $A_1$ to be the set of those $x\in A$ such that $1_{S'}\ast 1_A(x)\geq \tfrac{1}{2}\delta\abs{S'}$. In general, with $A_{i-1},A_i$ defined, so
\[\langle 1_{A_{i-1}}, 1_{S'}\ast 1_{A_{i}}\rangle \geq 2^{-i}\delta\abs{A}\abs{S'},\]
we take $A_{i+1}$ to be the set of $x\in A_{i-1}$ such that $1_{S'}\ast 1_{A_{i}}(x)\geq 2^{-i-1}\delta\abs{S'}$. This completes the construction; note that this has the property that $A_i\subseteq A_j$ whenever $i\geq j$ and $i\equiv j\pmod{2}$.

With $k\geq 1$ as in the assumptions we let $S$ be the set of those $x\in S'$ such that $1_{A_k}\circ 1_{A_{k-1}}(x)\geq 2^{-k-1}\delta\abs{A}$, so that
\begin{equation}\label{eq-aux1}
\langle 1_S, 1_A\circ 1_A\rangle \geq \langle 1_{S}, 1_{A_k}\circ 1_{A_{k-1}}\rangle \gg 2^{-O(k)}\delta\abs{A}\abs{S'}.
\end{equation}
Lemma~\ref{lem-innbound} implies
\[\delta^8\abs{S'}^6\ll 2^{O(k)}\kappa^{5/3}\abs{A}^{14/3}\norm{1_{S}\circ 1_{S}}_3.\]
Note that
\[\norm{1_{A_i}\circ 1_{S}}_{3/2} \leq \abs{A}\abs{S'}^{2/3}\leq (\tau\kappa)^{-O(1)}\abs{A}^{5/3},\]
and since \eqref{eq-aux1} implies via H\"{o}lder's inequality that $\norm{1_{A_k}\circ 1_S}_{3/2}\gs 2^{-O(k)}\kappa^{O(1)}\abs{A}^{5/3}$, the pigeonhole principle implies there exists some $1\leq i\leq k$ such that
\[\|1_{A_{i-1}}\circ 1_{S}\|_{3/2}\leq (\tau\kappa)^{-O(1/k)}\norm{1_{A_i}\circ 1_{S}}_{3/2}.\]
Furthermore, the nested nature of our $A_i$ ensures that $1_{A_i}\circ 1_{A}(x)\geq 2^{-O(k)}\delta\abs{A}$ for all $x\in S$, whence
\[\norm{1_{S}\circ 1_{S}}_3 \ls 2^{O(k)}\kappa^{1/3}\delta^{-1}\abs{A}^{-1/3}\norm{1_{A_i}\circ 1_{S}}_{3/2}.\]
It follows that
\[\delta^8\abs{S'}^6\ls 2^{O(k)}\kappa^{2}\delta^{-1}\abs{A}^{13/3}\norm{1_{A_i}\circ 1_{S}}_{3/2}.\]
The claim now follows, recalling $\abs{S'}\gs \delta^{-2}\tau\abs{A}$, setting $A_2=A_i$ and $A_1=A_{i-1}$.
\end{proof}

\section{Bounding $\norm{1_A\circ 1_S}_{3/2}$}\label{sec-final}

In this section we will prove two bounds on $\norm{1_A\circ 1_S}_{3/2}$ (where, as usual, $S$ is some upper level set of $1_A\circ 1_A$) which in certain regimes go beyond the `trivial' bound of \eqref{eq-holderbound}.

Many of the manipulations here may appear arbitrary and ad hoc; this is because they are. The particular shape of these arguments was arrived at through a blend of intuition and trial and error, and there seems to be nothing canonical or natural about some of the choices made. If the reader, encountering a particular inequality, thinks ``why was the choice made to bound it this way, rather than this other way'', we can only reassure them that we tried every way that occurred to us, but only the path described here led to any improvement at the end.

We believe that the ad hoc nature of the proofs in this section are a reflection of our lack of understanding of control; there should be a more natural way of producing stronger non-trivial bounds for $\norm{1_A\circ 1_S}_{3/2}$, but exactly what form this should take eludes us.

We also cannot rule out the possibility that an elementary proof similar to those below, with only a slightly different path, could produce superior exponents. We have tried to make sure that at least the exponents produced in Theorem~\ref{th-main} are a `local maximum' of these techniques, but the large search space of possible elementary arguments makes this difficult to guarantee.

\subsection{A bound useful for $\abs{A-A}$}
The first lemma gives an improvement which is useful when bounding $\abs{A-A}$ (in conjunction with Lemma~\ref{lem-diffapp}).
\begin{lemma}\label{lem-diffbound}
If $A$ has control $\kappa$ and $\abs{A-A}\leq K\abs{A}$ then, for any $\delta>0$, if
\[S=\{ x: 1_A\circ 1_A(x) \geq \delta \abs{A}\},\]
then
\[\norm{1_A\circ 1_S}_{3/2}\ls_{\delta\kappa}  \delta^{-3/10}\kappa^{13/60}K^{11/60}\abs{A}^{7/15}\abs{S}^{19/30}\norm{1_A\circ 1_A}_{12/5}^{2/5}.\]
In particular, if $E(A)=\tau\abs{A}^3$ then
\[\norm{1_A\circ 1_S}_{3/2}\ls_{\delta\kappa} \tau^{1/10}\delta^{-3/10}\kappa^{17/60}K^{11/60}\abs{A}^{31/30}\abs{S}^{19/30}.\]
\end{lemma}
\begin{proof}
Let $\eta>0$ be arbitrary and $T=\{ x: 1_A\circ 1_S(x)\geq \eta \abs{A}\}$. By definition
\[\eta \abs{A}\abs{T} \leq \langle 1_A\circ 1_S, 1_T\rangle = \langle 1_S, 1_A\circ 1_T\rangle\]
and hence, by the Cauchy-Schwarz inequality,
\[\eta^2\abs{A}^2\abs{T}^2\abs{S}^{-1}\leq \langle 1_S, (1_A\circ 1_T)^2\rangle=\sum_{y_1,y_2\in T}\sum_{x\in S}1_A(x+y_1)1_A(x+y_2).\]
The sum here is restricted to pairs $y_1,y_2$ such that $y_1-y_2\in A-A$. Since
\begin{align*}
\sum_{y_1,y_2\in T}\brac{\sum_{x\in S}1_A(x+y_1)1_A(x+y_2)}^2
&= \sum_{x_1,x_2\in S}\brac{ \sum_{y\in T}1_A(x_1+y)1_A(x_2+y)}^2\\
&\leq \langle 1_S\circ 1_S, (1_A\circ 1_A)^2\rangle
\end{align*}
another application of the Cauchy-Schwarz inequality yields
\[\eta^4\abs{A}^{4}\abs{T}^4\abs{S}^{-2}\leq \langle 1_T\circ 1_T, 1_{A-A}\rangle\langle 1_S\circ 1_S, (1_A\circ 1_A)^2\rangle.\]
We first note that, using H\"{o}lder's inequality and the control assumption,
\begin{align*}
\langle 1_T\circ 1_T, 1_{A-A}\rangle
&\leq \eta^{-1}\abs{A}^{-1}\langle 1_A \circ 1_S, 1_T\ast 1_{A-A}\rangle\\
&\leq \eta^{-1}\abs{A}^{-1}\norm{1_A\circ 1_T}_{3}\norm{1_S\ast 1_{A-A}}_{3/2}\\
&\leq \kappa^{1/3}\eta^{-1}\abs{A}^{-1/3}\abs{T}^{2/3}\norm{1_S\ast 1_{A-A}}_{3/2}.
\end{align*}
By H\"{o}lder's inequality
\begin{align*}
\norm{1_S\ast 1_{A-A}}_{3/2}
&\leq K^{1/2}\abs{A}^{1/2}\abs{S}^{1/2}\norm{1_S\ast 1_{A-A}}_{3}^{1/2}\\
&\leq K^{1/2}\delta^{-1/2}\abs{S}^{1/2}\norm{1_A\circ 1_A\ast 1_{A-A}}_{3}^{1/2}\\
&\ls_\kappa 
\kappa^{1/4}\delta^{-1/2}K^{11/12}\abs{S}^{1/2}\abs{A}^{7/6},
\end{align*}
where we have used Lemma~\ref{lem-auxcont} and the `trivial' H\"{o}lder bound (as in \eqref{eq-holderbound}) of $\norm{1_A\ast 1_{A-A}}_{3/2}\leq \kappa^{1/6}K^{5/6}\abs{A}^{5/3}$ to bound 
\begin{align*}
\norm{1_A\circ 1_A\ast 1_{A-A}}_{3}
&\ls_\kappa \kappa^{1/3}\abs{A}^{2/3}\norm{1_A\ast 1_{A-A}}_{3/2}+\kappa^{100}K^{1/3}\abs{A}^{5/3}\norm{1_A\ast 1_{A-A}}_\infty^{2/3}\\
&\leq \kappa^{1/3}\abs{A}^{2/3}\norm{1_A\ast 1_{A-A}}_{3/2}+\kappa^{100}K^{1/3}\abs{A}^{7/3}\\
&\leq \kappa^{1/2}K^{5/6}\abs{A}^{7/3}+\kappa^{100}K^{1/3}\abs{A}^{7/3}\\
&\ll \kappa^{1/2}K^{5/6}\abs{A}^{7/3}.
\end{align*} 
It follows that
\[\brac{(\eta\abs{A})^{3/2}\abs{T}}^{10/3}\leq  \kappa^{7/12}\delta^{-1/2}K^{11/12}\abs{S}^{5/2}\abs{A}^{11/6}\langle 1_S\circ 1_S, (1_A\circ 1_A)^2\rangle.\]
Summing over a dyadic range of $\eta$ (noting that the contribution to the left-hand side from those $\eta \leq (\kappa \delta)^{100}$, say, is negligible) we deduce that 
\[\norm{1_A\circ 1_S}_{3/2}^5\ls \kappa^{7/12}\delta^{-1/2}K^{11/12}\abs{S}^{5/2}\abs{A}^{11/6}\langle 1_S\circ 1_S, (1_A\circ 1_A)^2\rangle.\]
We now bound $\langle 1_S\circ 1_S, (1_A\circ 1_A)^2\rangle$. Let $\nu>0$ and $U=\{ x: 1_A\circ 1_A(x)\in[\nu,2\nu)\abs{A}\}$. The contribution to the inner product from $U$ is
\begin{align*}
&\ll \nu^2\abs{A}^2\langle 1_S, 1_S\ast 1_U\rangle\\
&\ll \delta^{-1}\nu^2\abs{A}\langle 1_A\circ 1_A, 1_S\ast 1_U\rangle\\
&\leq \delta^{-1}\nu^2\abs{A}\norm{1_A\circ 1_S}_{3}\norm{1_A\ast 1_U}_{3/2}\\
&\leq \delta^{-1}\nu^2\kappa^{1/3}\abs{A}^{5/3}\abs{S}^{2/3}\norm{1_A\ast 1_U}_{3/2}.
\end{align*}
Using once again the `trivial' H\"{o}lder bound $\norm{1_A\ast 1_U}_{3/2}\leq \kappa^{1/6}\abs{A}^{5/6}\abs{U}^{5/6}$ this is at most
\[\delta^{-1}\kappa^{1/2}\abs{A}^{1/2}\abs{S}^{2/3}\brac{\nu^2\abs{A}^2\abs{U}^{5/6}}.\]
The bracketed expression is at most $\norm{1_A\circ 1_A}_{12/5}^{2}$, and hence summing over a dyadic range of $\nu$ we deduce 
\[\langle 1_S\circ 1_S, (1_A\circ 1_A)^2\rangle\ls \delta^{-1}\kappa^{1/2}\abs{A}^{1/2}\abs{S}^{2/3}\norm{1_A\circ 1_A}_{12/5}^{2}.\]
Rearranging yields the first result. The second follows from H\"{o}lder's inequality, since 
\[\norm{1_A\circ 1_A}_{12/5}\leq \norm{1_A\circ 1_A}_{2}^{1/2}\norm{1_A\circ 1_A}_{3}^{1/2}\leq \tau^{1/4}\kappa^{1/6}\abs{A}^{17/12}.\]
\end{proof}
\subsection{A bound useful for $E(A)$}
We now present the second type of improvement, used for improving the bounds on the additive energy $E(A)$ -- this is more technical, and we require slightly more information. This is not a significant barrier to applications, however, and indeed Lemma~\ref{lem-enapp} was constructed with the hypotheses of Lemma~\ref{lem-key} in mind. At first glance the reader should heuristically take $S=S'$ and $A_1=A_2=A$ in the below, and $\nu=\delta\abs{S}/\abs{A}$. 
\begin{lemma}\label{lem-key}
Suppose $A$ has control $\kappa$ and $A_1,A_2\subseteq A$ and $S,S'$ are such that
\begin{enumerate}
\item $1_{A_2}\circ 1_A(x)\geq \delta\abs{A}$ for all $x\in S$,
\item $1_{A_1}\circ 1_{S'}(x)\geq \nu\abs{A}$ for all $x\in A_2$, and
\item $\norm{1_{A_1}\circ 1_{S}}_{3/2}\leq L\norm{1_{A_2}\circ 1_{S}}_{3/2}$.
\end{enumerate} Then
\[\norm{1_{A_2}\circ 1_{S}}_{3/2} \ls_\kappa L^{1/7}\nu^{-8/7}\kappa^{4/7}\delta^{-2/7}\abs{A}^{8/21}\abs{S'}^{4/7}\abs{S}^{5/7}+\kappa^{99}\abs{A}^{4/3}\abs{S}^{1/3}.\]
\end{lemma}
\begin{proof}
Let $\eta>0$ and 
\[T = \{ x : 1_{A_2}\circ 1_{S}(x)\geq \eta\abs{A}\}.\]
We have
\[\eta\abs{A}\abs{T} \leq \langle 1_T, 1_{A_2}\circ 1_{S}\rangle = \langle 1_{A_2}, 1_{S}\ast 1_T\rangle\]
and hence
\[\nu\eta \abs{A}^2\abs{T} \leq \langle 1_{A_2}, (1_{A_1}\circ 1_{S'})(1_{S}\ast 1_T)\rangle=\sum_{\substack{z\in S\\ x\in S'}} \sum_{y\in T}1_{A_2}(z+y)1_{A_1}(z+y+x).\]
Therefore, by the Cauchy-Schwarz inequality,
\begin{align*}
\nu^2\eta^2\abs{A}^4\abs{T}^2\abs{S'}^{-1}\abs{S}^{-1}
&\leq \sum_{y_1,y_2\in T}\sum_{z\in S}1_{A_2}(z+y_1)1_{A_2}(z+y_2)\sum_{x\in S'}1_{A_1}(z+y_1+x)1_{A_1}(z+y_2+x)\\
&\leq \sum_{y_1,y_2\in T}1_{A_1}\circ 1_{A_1}(y_1-y_2)\sum_{z\in S}1_{A_2}(z+y_1)1_{A_2}(z+y_2).
\end{align*}
By the Cauchy-Schwarz inequality once again, we deduce that
\begin{equation}\label{eqa}
\nu^4\eta^4\abs{A}^8\abs{T}^4\abs{S'}^{-2}\abs{S}^{-2}\leq \langle 1_T\circ 1_T, (1_{A_1}\circ 1_{A_1})^2\rangle \langle 1_{S}\circ 1_{S}, (1_{A_2}\circ 1_{A_2})^2\rangle.
\end{equation}
To bound the first factor, we let $\rho>0$ and $U = \{ x: 1_{A_1}\circ 1_{A_1}(x)\in [\rho\abs{A},2\rho\abs{A})\}$, so that
\begin{align*}
\langle 1_T\circ 1_T, (1_{A_1}\circ 1_{A_1})^2\rangle
&\ll \sum_\rho \rho^2\abs{A}^2\langle 1_T, 1_T\ast 1_{U}\rangle\\
&\ll \eta^{-1}\abs{A}\sum_\rho \rho^2\langle 1_{A}\circ 1_S, 1_T\ast 1_{U}\rangle\\
&\ll \eta^{-1}\abs{A}\sum_\rho \rho^2\norm{1_A\circ 1_T}_3\norm{1_S\ast 1_U}_{3/2}\\
&\ll \kappa^{1/3}\eta^{-1}\abs{A}^{5/3}\abs{T}^{2/3}\sum_\rho \rho^2\norm{1_S\ast 1_U}_{3/2},
\end{align*}
where the sum over $\rho$ is over a dyadic range $\rho=2^{-k}$ for $k\geq 0$. By H\"{o}lder's inequality
\[\norm{1_{S}\ast 1_{U}}_{3/2}\leq \norm{1_{S}\ast 1_{U}}_3^{1/2}(\abs{S}\abs{{U}})^{1/2}\]
and by Lemma~\ref{lem-auxcont}
\begin{align*}
\norm{1_{S}\ast 1_{U}}_3
&\ll \rho^{-1}\abs{A}^{-1}\norm{1_{A_1}\circ 1_{A}\ast 1_{S}}_3\\
&\ls_\kappa \kappa^{1/3}\rho^{-1}\abs{A}^{-1/3}\norm{1_{A_1}\ast 1_{S}}_{3/2}+\kappa^{100}\rho^{-1}\abs{A}\abs{S}^{1/3}.
\end{align*}
It follows that $\langle 1_T\circ 1_T, (1_{A_1}\circ 1_{A_1})^2\rangle$ is at most
\[\kappa^{1/3}\eta^{-1}\abs{T}^{2/3}\abs{S}^{1/2}\brac{\kappa^{1/6}\norm{1_{A_1}\ast 1_{S}}_{3/2}^{1/2}+\kappa^{50}\abs{A}^{2/3}\abs{S}^{1/6}}\sum_\rho (\rho\abs{A})^{3/2}\abs{U}^{1/2}.\]
The sum over $\rho$ is at most $\ls \kappa^{1/2}\abs{A}^2$ (noting that we can restrict the sum to those $\rho\geq \kappa^{100}$, say, since the contribution from very small $\rho$ is negligible), whence
\[\langle 1_T\circ 1_T, (1_{A_1}\circ 1_{A_1})^2\rangle \ls\kappa^{5/6}\eta^{-1}\abs{A}^2\abs{T}^{2/3}\abs{S}^{1/2}\brac{\kappa^{1/6}\norm{1_{A_1}\ast 1_{S}}_{3/2}^{1/2}+\kappa^{50}\abs{A}^{2/3}\abs{S}^{1/6}}.\]
Combining this with \eqref{eqa}, we deduce
\[\brac{\eta^{3/2}\abs{A}^{3/2}\abs{T}}^{10/3} \ls\]
\[ \nu^{-4}\kappa^{5/6}\abs{A}^{-1}\abs{S'}^{2}\abs{S}^{5/2}\brac{\kappa^{1/6}\norm{1_{A_1}\ast 1_{S}}_{3/2}^{1/2}+\kappa^{50}\abs{A}^{2/3}\abs{S}^{1/6}}\langle 1_{S}\circ 1_{S}, (1_{A}\circ 1_{A})^2\rangle.\]
Summing this over $\eta$ in a dyadic range we deduce that the same upper bound holds for $\norm{1_{A_2}\circ 1_{S}}_{3/2}^{5}$. We then bound the final factor using H\"{o}lder's inequality by
\[\norm{1_A\circ 1_A}_3^2\norm{1_{S}\circ 1_{S}}_3\leq\kappa^{2/3} \abs{A}^{8/3}\norm{1_{S}\circ 1_{S}}_3.\]
Finally, by Lemma~\ref{lem-auxcont},
\begin{align*}
\norm{1_{S}\circ 1_{S}}_3
&\leq \delta^{-1}\abs{A}^{-1}\norm{1_{A_2}\circ 1_A\circ 1_S}_3\\
&\ls_\kappa \delta^{-1}\kappa^{1/3}\abs{A}^{-1/3}\norm{1_{A_2}\circ 1_S}_{3/2}+\delta^{-1}\kappa^{100}\abs{A}\abs{S}^{1/3}.\end{align*}
Combining these inequalities we have bounded $\norm{1_{A_2}\circ 1_S}_{3/2}^5$ above by (noting that $\norm{1_{A_1}\ast 1_S}_{3/2}=\norm{1_{A_1}\circ 1_S}_{3/2}\leq L\norm{1_{A_2}\circ 1_S}_{3/2}$)
\[\ls \delta^{-1}\nu^{-4}\kappa^{3/2}\abs{A}^{4/3}\abs{S'}^{2}\abs{S}^{5/2}\brac{L^{1/2}\kappa^{1/6}\norm{1_{A_2}\circ 1_{S}}_{3/2}^{1/2}+\kappa^{50}\abs{A}^{2/3}\abs{S}^{1/6}}\times\]
\[\brac{\kappa^{1/3}\norm{1_{A_2}\circ 1_S}_{3/2}+\kappa^{100}\abs{A}^{4/3}\abs{S}^{1/3}}.\]
In particular, either
\[\norm{1_{A_2}\circ 1_S}_{3/2}\ls \kappa^{99}\abs{A}^{4/3}\abs{S}^{1/3}\]
or
\[\norm{1_{A_2}\circ 1_{S}}_{3/2}^{5} \ls L\nu^{-4}\kappa^{2}\delta^{-1}\abs{A}^{4/3}\abs{S'}^2\abs{S}^{5/2}\norm{1_{A_1}\circ 1_{S}}_{3/2}^{1/2} \norm{1_{A_2}\circ 1_{S}}_{3/2}.\]
The claim follows after simplifying this inequality.
\end{proof}

\subsection{Completing the proof of Theorem~\ref{th-main}}

Finally we bring together the various inequalities to prove Theorem~\ref{th-main}. (It is worth noting that the lower bound for $\abs{A-A}$ uses the upper bound for $E(A)$ -- in particular, an improvement for the latter immediately implies an improvement for the former.)

\begin{proof}[Proof of Theorem~\ref{th-main}]
We first note that the lower bound
\[\abs{A+A}\gs \kappa^{11/19}\abs{A}\]
is an immediate consequence of Lemma~\ref{lem-sumapp} and \eqref{eq-holderbound}, as previously noted.

To prove the energy bound, we apply Lemma~\ref{lem-enapp} to produce $A_1,A_2,S,S',\delta$ as in the statement with
\[\tau^6\abs{A}^{5/3}\ll 2^{O(k)}\kappa^2\delta^3\norm{1_{A_2}\circ 1_S}_{3/2}.\]
We may now apply Lemma~\ref{lem-key} with $\nu \approx 2^{-O(k)}\delta\abs{S'}/\abs{A}$ and $L\ll \kappa^{-O(1/k)}$ (noting that we can assume $\tau \geq \kappa$, say, or we are immediately done). This yields
\[\norm{1_{A_2}\circ 1_{S}}_{3/2} \ls 2^{O(k)}\kappa^{4/7-O(1/k)}\delta^{-10/7}\abs{A}^{32/21}\abs{S'}^{-4/7}\abs{S}^{5/7}+\kappa^{99}\abs{A}^{4/3}\abs{S}^{1/3}.\]
Using the bounds $\abs{S'}\gs \delta^{-2}\tau \abs{A}$ and $\abs{S}\ls 2^{O(k)}\delta^{-2}\tau\abs{A}$ this becomes
\[\norm{1_{A_2}\circ 1_{S}}_{3/2} \ls 2^{O(k)}\kappa^{4/7-O(1/k)}\delta^{-12/7}\tau^{1/7}\abs{A}^{5/3},\]
and hence
\[\tau^{6-1/7}\ls 2^{O(k)}\kappa^{2+4/7-O(1/k)}\delta^{3-12/7}.\]
Using $\delta \ls \kappa\tau^{-1}$ and simplifying this expression, we have
\[\tau^{50/7}\ls 2^{O(k)}\kappa^{27/7-O(1/k)}.\]
Choosing $k\asymp \epsilon^{-1}$ yields $\tau \ll_\epsilon \kappa^{27/50-\epsilon}$ as required.

Finally, to bound $\abs{A-A}=K\abs{A}$, we apply Lemma~\ref{lem-diffapp} to yield
\[\abs{A}^{5/3} \ll \kappa^{4/3}K^{5/3}\norm{1_A\circ 1_S}_{3/2},\]
where
\[S = \{ x : 1_A\circ 1_A(x) \geq (2K)^{-1}\abs{A}\}.\]
Lemma~\ref{lem-diffbound} (with $\delta=1/2K$ and $\abs{S}\leq K\abs{A}$) yields 
\[\norm{1_A\circ 1_S}_{3/2}\ls \tau^{1/10}\kappa^{17/60}K^{67/60}\abs{A}^{5/3},\]
whence
\[1\ls \tau^{1/10}\kappa^{97/60}K^{167/60}.\]
Using the bound $\tau \ll_\epsilon \kappa^{27/50-\epsilon}$ yields
\[1\ll_\epsilon \kappa^{1253/750-\epsilon}K^{167/60},\]
and rearranging produces
\[K^{-1} \ll_\epsilon \kappa^{2506/4175-\epsilon}\]
as claimed.
\end{proof}
\bibliographystyle{plain}
\bibliography{refs} 

\begin{thebibliography}{10}

\bibitem{BaSz94}
Antal Balog and Endre Szemer\'edi.
\newblock A statistical theorem of set addition.
\newblock {\em Combinatorica}, 14(3):263--268, 1994.

\bibitem{BaWo17}
Antal Balog and Trevor~D. Wooley.
\newblock A low-energy decomposition theorem.
\newblock {\em Q. J. Math.}, 68(1):207--226, 2017.

\bibitem{BHR22}
Peter Bradshaw, Brandon Hanson, and Misha Rudnev.
\newblock Higher convexity and iterated second moment estimates.
\newblock {\em Electron. J. Combin.}, 29(3):Paper No. 3.6, 23, 2022.

\bibitem{Er77}
P.~Erd\H{o}s.
\newblock Problems in number theory and combinatorics.
\newblock In {\em Proceedings of the {S}ixth {M}anitoba {C}onference on
  {N}umerical {M}athematics ({U}niv. {M}anitoba, {W}innipeg, {M}an., 1976)},
  volume XVIII of {\em Congress. Numer.}, pages 35--58. Utilitas Math.,
  Winnipeg, MB, 1977.

\bibitem{Go01}
W.~T. Gowers.
\newblock A new proof of {S}zemer\'edi's theorem.
\newblock {\em Geom. Funct. Anal.}, 11(3):465--588, 2001.

\bibitem{He86}
N.~Hegyv\'ari.
\newblock On consecutive sums in sequences.
\newblock {\em Acta Math. Hungar.}, 48(1-2):193--200, 1986.

\bibitem{KoSh15}
S.~V. Konyagin and I.~D. Shkredov.
\newblock On sum sets of sets having small product set.
\newblock {\em Proc. Steklov Inst. Math.}, 290(1):288--299, 2015.

\bibitem{Ma16}
P.~Mazur.
\newblock Some results in set addition.
\newblock {\em PhD thesis, University of Oxford}, 2016.

\bibitem{MRS17}
Brendan Murphy, Oliver Roche-Newton, and Ilya~D. Shkredov.
\newblock Variations on the sum-product problem {II}.
\newblock {\em SIAM J. Discrete Math.}, 31(3):1878--1894, 2017.

\bibitem{Ol20}
K.~I. Olmezov.
\newblock Additive properties of slowly increasing convex sets.
\newblock {\em Mat. Zametki}, 108(6):851--867, 2020.

\bibitem{Ol21}
K.~I. Olmezov.
\newblock An elementary analog of the operator method in additive
  combinatorics.
\newblock {\em Mat. Zametki}, 109(1):117--128, 2021.

\bibitem{PaZa24}
J.~Pach and D.~Zakharov.
\newblock Ruzsa's problem on bi-{S}idon sets.
\newblock {\em arXiv:2409.03128}, 2024.

\bibitem{ReSc24}
Christian Reiher and Tomasz Schoen.
\newblock Note on the theorem of {B}alog, {S}zemer\'edi, and {G}owers.
\newblock {\em Combinatorica}, 44(3):691--698, 2024.

\bibitem{RSS20}
Misha Rudnev, Ilya~D. Shkredov, and Sophie Stevens.
\newblock On the energy variant of the sum-product conjecture.
\newblock {\em Rev. Mat. Iberoam.}, 36(1):207--232, 2020.

\bibitem{RuSt22}
Misha Rudnev and Sophie Stevens.
\newblock An update on the sum-product problem.
\newblock {\em Math. Proc. Cambridge Philos. Soc.}, 173(2):411--430, 2022.

\bibitem{Ru06}
I.~Z. Ruzsa.
\newblock Additive and multiplicative {S}idon sets.
\newblock {\em Acta Math. Hungar.}, 112(4):345--354, 2006.

\bibitem{Sc15}
Tomasz Schoen.
\newblock New bounds in {B}alog-{S}zemer\'edi-{G}owers theorem.
\newblock {\em Combinatorica}, 35(6):695--701, 2015.

\bibitem{ScSh11}
Tomasz Schoen and Ilya~D. Shkredov.
\newblock On sumsets of convex sets.
\newblock {\em Combin. Probab. Comput.}, 20(5):793--798, 2011.

\bibitem{Sh19}
George Shakan.
\newblock On higher energy decompositions and the sum-product phenomenon.
\newblock {\em Math. Proc. Cambridge Philos. Soc.}, 167(3):599--617, 2019.

\bibitem{Sh13}
I.~D. Shkredov.
\newblock Some new results on higher energies.
\newblock {\em Trans. Moscow Math. Soc.}, pages 31--63, 2013.

\bibitem{Sh15}
I.~D. Shkredov.
\newblock On sums of {S}zemer\'edi-{T}rotter sets.
\newblock {\em Proc. Steklov Inst. Math.}, 289(1):300--309, 2015.
\newblock Published in Russian in Tr. Mat. Inst. Steklova {\bf 289} (2015),
  318--327.

\bibitem{So09}
J\'ozsef Solymosi.
\newblock Bounding multiplicative energy by the sumset.
\newblock {\em Adv. Math.}, 222(2):402--408, 2009.

\bibitem{StdZ17}
Sophie Stevens and Frank de~Zeeuw.
\newblock An improved point-line incidence bound over arbitrary fields.
\newblock {\em Bull. Lond. Math. Soc.}, 49(5):842--858, 2017.

\bibitem{StWa22}
Sophie Stevens and Audie Warren.
\newblock On sum sets and convex functions.
\newblock {\em Electron. J. Combin.}, 29(2):Paper No. 2.18, 19, 2022.

\bibitem{SzTr83}
Endre Szemer\'edi and William~T. Trotter, Jr.
\newblock Extremal problems in discrete geometry.
\newblock {\em Combinatorica}, 3(3-4):381--392, 1983.

\bibitem{Xu21}
Boqing Xue.
\newblock Asymmetric estimates and the sum-product problems.
\newblock {\em Acta Arith.}, 198(3):289--311, 2021.

\end{thebibliography}
\end{document}